\documentclass[a4paper,12pt]{amsart}
\usepackage{amsmath, amscd, amssymb, mathrsfs, mathtools,fullpage, 
enumerate, svninfo,xcolor,enumitem}
\usepackage[all]{xy}
\usepackage{yfonts}
\usepackage{stmaryrd}
\usepackage{hyperref}

\title{$p$-adic $L$-functions for Hecke characters of totally imaginary fields}

\author{Guido Kings}
\address{Fakult\"at f\"ur Mathematik \\
Universit\"at Regensburg\\
93040 Regensburg\\
Germany}
\author{Johannes Sprang}
\address{
Fakultät für Mathematik\\
Universität Duisburg-Essen\\
45127 Essen\\
Germany}

\thanks{This research was supported by the DFG grant: SFB 1085 Higher invariants}

\theoremstyle{plain}
    \newtheorem{theorem}{Theorem}[section]
\newtheorem*{theorem*}{Theorem}
    \newtheorem{lemma}[theorem]{Lemma}
    \newtheorem{proposition}[theorem]{Proposition}
    \newtheorem{corollary}[theorem]{Corollary}
    \newtheorem*{corollary*}{Corollary}
    \newtheorem{notation}[theorem]{Notation}

\theoremstyle{remark}
    \newtheorem{remark}[theorem]{Remark}
    
     \newtheorem{example}[theorem]{Example}

\theoremstyle{definition}
    \newtheorem{definition}[theorem]{Definition}
    
\newenvironment{maintheorem}[1]{%
  \renewcommand\thetheorem{#1}
  \theorem
}{\endtheorem}

\DeclareMathOperator{\TSym}{TSym}
\DeclareMathOperator{\Sym}{Sym}
\DeclareMathOperator{\Hom}{Hom}

\DeclareMathOperator{\Spec}{Spec}

\DeclareMathOperator{\Spf}{Spf}

\DeclareMathOperator{\pr}{pr}

\DeclareMathOperator{\Char}{char}
\DeclareMathOperator{\Gal}{Gal}

\DeclareMathOperator{\Lie}{Lie}

\newcommand{\padic}{\text{\textup{p-adic}}}

\newcommand{\res}{\mathrm{res}}
\newcommand{\cor}{\mathrm{cor}}

\newcommand{\sG}{\mathscr{G}}

\newcommand{\sO}{\mathscr{O}}

\newcommand{\sP}{\mathscr{P}}

\newcommand{\cA}{\mathcal{A}}
\newcommand{\cB}{\mathcal{B}}

\newcommand{\cF}{\mathcal{F}}

\newcommand{\cI}{\mathcal{I}}

\newcommand{\cO}{\mathcal{O}}
\newcommand{\cP}{\mathcal{P}}

\newcommand{\cR}{\mathcal{R}}

\newcommand{\cU}{\mathcal{U}}

\newcommand{\fra}{\mathfrak{a}}
\newcommand{\frb}{\mathfrak{b}}
\newcommand{\frc}{\mathfrak{c}}

\newcommand{\frf}{\mathfrak{f}}

\newcommand{\frn}{\mathfrak{n}}

\newcommand{\frp}{\mathfrak{p}}
\newcommand{\frt}{\mathfrak{t}}

\newcommand{\frN}{\mathfrak{N}}
\newcommand{\frP}{\mathfrak{P}}

\newcommand{\bbA}{\mathbb{A}}
\newcommand{\CC}{\mathbb{C}}

\newcommand{\QQ}{\mathbb{Q}}
\newcommand{\RR}{\mathbb{R}}
\newcommand{\ZZ}{\mathbb{Z}}

\newcommand{\FF}{\mathbb{F}}
\newcommand{\NN}{\mathbb{N}}

\newcommand{\Q}{\mathbb{Q}}
\newcommand{\Z}{\mathbb{Z}}
\newcommand{\R}{\mathbb{R}}
\newcommand{\C}{\mathbb{C}}
\newcommand{\Qp}{{\QQ_p}}
\newcommand{\Zp}{{\ZZ_p}}

\newcommand{\pDiv}{\mathbf{pDiv}}

\newcommand{\HTDpairs}{\mathbf{HT}^\vee\mathbf{pairs}}
\newcommand{\Adic}{\mathbf{Adic}}
\newcommand{\Rigid}{\mathbf{Rigid}}

\newcommand{\CatCharVar}{\mathbf{CharVar}}

\newcommand{\grT}{T}
\newcommand{\CharVar}[1]{\widehat{#1}}

\newcommand{\dR}{\mathrm{dR}}
\newcommand{\rig}{\mathrm{rig}}
\newcommand{\fin}{\mathrm{fin}}
\newcommand{\ad}{\mathrm{ad}}
\newcommand{\an}{\mathrm{an}}

\newcommand{\prolim}{\varprojlim}

\newcommand{\Cloc}{C}
\newcommand{\Lloc}{F}
\newcommand{\Llocan}{F\text{-an}}
\newcommand{\isom}{\cong}

\newcommand{\id}{\mathrm{id}}

\newcommand{\ul}[1]{\underline{#1}}
\newcommand{\ol}[1]{\overline{#1}}
\DeclareMathOperator{\wP}{\widehat{\sP}}

\newcommand{\dashsum}{\sideset{}{'}{\sum}}

\newcommand{\EK}{EK}

\DeclareMathOperator{\LieA}{\partial(\cA)}

\def\endpiece{xxx}
\def\makeAlphabet[#1]{\expandafter\makeA#1,xxx,}
\def\makealphabet[#1]{\expandafter\makea#1,xxx,}
\def\makeA#1,{\def\temp{#1}\ifx\temp\endpiece\else%
\mkbb{#1}\mkfrak{#1}\mkbf{#1}\mkcal{#1}\mkscr{#1}\expandafter\makeA\fi}%
\def\makea#1,{\def\temp{#1}\ifx\temp\endpiece\else\mkfrak{#1}\mkbf{#1}\expandafter\makea\fi}%
\def\mkbb#1{\expandafter\def\csname bb#1\endcsname{\mathbb{#1}}}%
\def\mkfrak#1{\expandafter\def\csname fr#1\endcsname{\mathfrak{#1}}}
\def\mkbf#1{\expandafter\def\csname b#1\endcsname{\mathbf{#1}}}
\def\mkcal#1{\expandafter\def\csname c#1\endcsname{\mathcal{#1}}}
\def\mkscr#1{\expandafter\def\csname s#1\endcsname{\mathscr{#1}}}
\def\makeop[#1]{\xmakeop#1,xxx,}
\def\mkop#1{\expandafter\def\csname #1\endcsname{{\mathrm{#1}}}} %
\def\xmakeop#1,{\def\temp{#1}\ifx\temp\endpiece\else\mkop{#1}\expandafter\xmakeop\fi}%

\makeop[Mic,Mod,im]

\newcommand{\Gmf}{\widehat{\mathbb{G}}_m}
\newcommand{\OCp}{\sO_{\CC_p}}
\newcommand{\Cp}{{\CC_p}}

\newcommand{\eval}{\mathrm{eval}}

\newcommand{\Local}{\mathrm{Local}}
\newcommand{\cond}{\mathrm{cond}}

\AtBeginDocument{%
   \def\MR#1{}
}

\let\Re\relax
\DeclareMathOperator{\Re}{Re} 

\numberwithin{equation}{section}

\begin{document}

\begin{abstract}We construct $p$-adic $L$-functions interpolating critical $L$-values of algebraic Hecke characters for arbitrary unramified primes $p$ and any totally imaginary field. For non-ordinary primes, the only previously known case was that of imaginary quadratic extensions of $\QQ$. One of the main
ingredients is a new  $p$-adic Fourier theory relating generic fibers of $p$-divisible groups to a general class of character varieties. Combining this with equivariant cohomology classes constructed in a previous paper allows us to construct the $p$-adic $L$-function.
\end{abstract}

\setcounter{tocdepth}{2}
\maketitle

\tableofcontents

`\section*{Introduction}

The idea to construct $p$-adic $L$-functions as functions on the $p$-adic integers interpolating the values of Dirichlet $L$-functions at negative integers goes back to Kubota and Leopoldt. Of course, such a construction makes only sense, if algebraicity properties of these special values are known.

The Deligne conjecture on critical special values of $L$-functions of motives predicts these values to be algebraic up to explicit periods. These algebraic integers are in general expected  to encode deep arithmetic information about the underlying motive. A well-established method to connect the $L$-values with arithmetic properties consists in constructing a $p$-adic $L$-function  interpolating these critical values and relate  them via Iwasawa theory to arithmetic invariants. 

An important class of $L$-functions for which the algebraicity of critical values is known are  $L$-functions of algebraic Hecke characters. In this case, critical Hecke $L$-values can occur only when the underlying number field is either totally real or totally imaginary and contains a CM field. In the totally real case, the algebraicity of the $L$-values was established by Siegel and Klingen, and the corresponding $p$-adic $L$-functions were first constructed independently by Barsky, Cassou-Noug{\`e}s and Deligne–Ribet (alphabetic order). Through the Iwasawa main conjecture established by Wiles, these $p$-adic $L$-functions are related to arithmetic information as the  class numbers in $\mathbb{Z}_p$-extensions of the totally real field. 

The situation for totally imaginary fields is considerably more difficult for two reasons: First, for general totally imaginary fields, a complete proof of the algebraicity of critical Hecke $L$-values has only been obtained relatively recently. Second, the $p$-adic behavior of these values depends on the splitting of $p$ in the maximal CM subfield, leading to the distinction between ordinary and non-ordinary primes. For ordinary primes, $p$-adic $L$-functions have been constructed by Katz for CM fields in \cite{Katz-CM}, and more recently in \cite{Kings-Sprang} for general totally imaginary fields containing a CM field. The non-ordinary case is substantially more difficult and only in the case of imaginary quadratic fields one has  some results. Here non-ordinary primes are precisely the non-split primes, and constructions of the corresponding $p$-adic $L$-functions were given in the work of Katz \cite{katz_formalgroups} (completed by \cite{AB} and \cite{Berger2025}), Boxall \cite{Boxall}, Schneider–Teitelbaum \cite{ST} and Andreatta–Iovita \cite{AI}. Not a single example of a $p$-adic $L$-function for non-ordinary primes interpolating critical Hecke $L$-values of totally imaginary fields has been constructed beyond the imaginary quadratic case, to the best of our knowledge.  
%

In this paper we construct a $p$-adic Hecke $L$-function for arbitrary primes $p$ interpolating all critical algebraic Hecke characters of any totally imaginary field $L$ whose infinity type is supported in a CM type $\Sigma$:

\begin{maintheorem}{A}[see Theorem \ref{thmA}]
There is a 
a  rigid analytic function $L_p(\cdot)$ 
with the following interpolation property:
For every critical algebraic Hecke character $\chi$ of conductor dividing $\frf p^\infty$ and infinity type $-\alpha\in \ZZ^\Sigma$ with $\alpha\geq \underline{1}$, we have
\[
	\frac{L_p(\chi)}{\Omega_p^{\alpha}} = (\mbox{explicit local factors})\frac{L_\frf(\chi,0)}{\Omega^{\alpha}}.
\] 
where $L_\frf(\chi,0)$ is the $L$-function of $\chi$ at $0$ and  $\Omega$ (resp. $\Omega_p$) are ($p$-adic) periods of an abelian variety $\cA$ with $\cA(\CC)=\CC^\Sigma/\sO_L$. 
\end{maintheorem}

We refer the reader to \ref{thmA} for a more precise formulation of the theorem. With considerably more technical efforts, one should be able to treat also Hecke characters of $\infty$-type $\beta-\alpha$ for a fixed $\beta$. The case of families with varying $\beta$ and $\alpha$ seems currently not to be in reach with our techniques, although we hope to come back to this situation in future work. 

Our point of view of the $p$-adic $L$-function $ L_p(\cdot)$ is motivated by an analogy with Tate's approach to complex Hecke $L$-functions: In his approach, the $L$-function is seen as a (meromorphic) function on the complex manifold parametrizing all Hecke characters of a fixed conductor. Similarly,  
$p$-adic Hecke characters of conductor dividing $p^\infty \frf$ of a number field $L$ are parametrized by the $\Cp$-points of a rigid analytic variety $\widehat{G(p^\infty\frf)}$, see  Section \ref{sec:Hecke-character-variety} for more details.  The $p$-adic Hecke characters whose infinity type is supported in a CM type $\Sigma$ are singled out by an analyticity condition depending on $\Sigma$, which gives rise to a subvariety $\widehat{G(p^\infty\frf)}_{W(\Sigma)}$ of $\widehat{G(p^\infty\frf)}$. The function $L_p(\cdot)$ lives on  $\widehat{G(p^\infty\frf)}_{W(\Sigma)}$.

The construction of $L_p(\cdot)$ is again reminiscent of the approach by Tate, which mixes the additive and multiplicative structure of the number field $L$. This includes the  Fourier-transform on the adeles,  as well as the Mellin-transform on the ideles. In our construction the additive theory is related to Eisenstein series.   More precisely, one has first to use $p$-adic Fourier theory on $\sO_L\otimes \Z_p$ and then restrict to $(\sO_L\otimes \Z_p)^{\times}$, whose quotient by the one-units $\sO_\frf^{\times}$ in turn parametrizes the connected components of the Galois group $\Gal(L(p^{\infty}\frf)/L)=G(p^{\infty}\frf)$.

Therefore we develop in the first part of the paper a general $p$-adic Fourier theory for certain abelian $p$-adic Lie groups $A$ and their $W$-analytic character variety $\CharVar{A}_W$ depending on a subspace $W$ of the cotangent space. These character varieties are a natural generalization of the ones appearing in \cite{ST}. Let $C$ be a complete non-Archimedean extension of $\Qp$. We show:

\begin{maintheorem}{B}[see Theorem \ref{thm:FT-abelian}]
	Let $A$ be a compact topologically finitely generated abelian $p$-adic Lie group, $\fra=\Lie(A)$ and $W\subseteq \Hom_{\Qp}(\fra,F)$. The Fourier transform induces an isomorphism of topological $C$-algebras
	\[
		\cF_W\colon D_W(A,\Cloc) \xrightarrow{\sim} \cO(\CharVar{A}_W/\Cloc), \quad \lambda\mapsto \cF_\lambda|_{\CharVar{A}_W},
	\]
where $D_W(A,\Cloc)$ are the $W$-analytic distributions and $\cO(\CharVar{A}_W/\Cloc)$ are the rigid analytic functions on $\CharVar{A}_W$.
\end{maintheorem}

In order to construct the function $L_p(\cdot)$, it remains to construct a rigid analytic function on the character variety $\widehat{G(p^\infty\frf)}_{W(\Sigma)}$. The main new idea here is to use the classification of $p$-divisible groups over $\OCp$ by Scholze-Weinstein \cite{ScholzeWeinstein}, to parametrize the $W$-analytic character varieties of a free $\Z_p$-module of finite rank $T$. 
For any $p$-divisible group we associate  the character variety $\widehat{T}_W$ on the Tate module $T=T_pG^\vee$ of the dual $p$-divisible group $G^{\vee}$, classifying all (locally) $W$-analytic characters on $T$, where $W=\Lie(G)$.  We prove:

\begin{maintheorem}{C}[see Theorem \ref{thm:UniformizationOfCharVar}]
There is an isomorphism of rigid analytic varieties over $\Cp$ 
\[
	\widehat{T}_W \xrightarrow{\sim} G_\eta^{rig}
\]
between the $W$-analytic character variety of $T=T_pG^\vee$ and the generic fiber of the $p$-divisible group $G$. Furthermore, the Fourier transform induces an isomorphism 
\begin{equation*}
    \cF_{W}\colon D_{W}(\grT,\Cp) \xrightarrow{\sim} \cO(\widehat{T}_W), 
\end{equation*}
between the space of $W$-analytic distributions on $T$ and the rigid analytic functions on the character variety $\widehat{T}_W$.
\end{maintheorem}
This theorem appeared already in our preprint \cite{KS24}, whose content is generalized and incorporated in this paper. We remark that a related Fourier theory has been developed independently by Graham, van Hoften, Howe \cite{GHH}.

In the second part of this paper, we apply the above result to the $p$-divisible group $A[p^\infty]$ of an abelian variety of CM type $\Sigma$ with CM by a totally imaginary field $L$. It turns out that we can relate the $p$-divisible group $A[p^\infty]$ to the subvariety $\widehat{G(p^\infty\frf)}_{W(\Sigma)}$ of the $p$-adic Hecke character variety $\widehat{G(p^\infty \frf)}$. In particular, this subvariety contains all critical Hecke characters with infinity type supported in $\Sigma$. Combining this with our earlier  results  \cite{Kings-Sprang}, which give algebraic functions on the $p$-divisible group $A[p^\infty]$, allows to construct a rigid analytic function $L_p\in \cO(\widehat{G(p^\infty \frf)})_{W(\Sigma)}$.

\part{$p$-adic Fourier theory}
We develop a $p$-adic Fourier transformation for rigid analytic functions on $p$-divisible groups in full generality. The main observation is that suitably defined  character varieties naturally fit into the classification scheme of Scholze-Weinstein \cite{ScholzeWeinstein} of $p$-divisible groups. Our theory can be seen as a generalization of the work of Schneider-Teitelbaum \cite{ST}.
\section{$W$-analytic character groups}

Let $ \Cloc\subset \C_p$ be a field extension of $\Q_p$ which is complete with respect to a non-archimedean valuation $|\cdot|$, extending the valuation on $\Qp$ and $\Lloc$ a complete extension of $\Q_p$ with $\Q_p\subset \Lloc\subset \Cloc$. Let $\grT$ be a free $\Zp$-module of finite rank. In the application in Part II, $T$ will be the  Tate module of the dual of an abelian variety with CM.  We will view $\grT$ as a $\Qp$-analytic group and denote by 
\begin{equation*}
    C^{\an}(\grT,\Cloc)
\end{equation*}
the space of $\Cloc$-valued locally $\Qp$-analytic functions on $\grT$. 
\begin{definition}
Let
\begin{equation*}
    \CharVar{\grT}(\Cloc):=\{\chi\in \Hom(\grT,\Cloc^{\times})\mid \chi\in C^{\an}(\grT,\Cloc)\}
\end{equation*}
be the group of locally $\Qp$-analytic $\Cloc$-valued characters on $\grT$. 
\end{definition}
The Lie algebra $\frt$ of $\grT$ is simply $\grT\otimes_{\Zp}\Qp$. Let us write $\frt^\vee:=\Hom_{\Qp}(\frt,\Qp)$ for the co-Lie algebra. We get the differential
\[
	d\colon C^{\an}(\grT,\Cloc)\to C^{\an}(\grT,\Cloc)\otimes_{\Qp} \frt^\vee.
\]
We consider subgroups of $C^{\an}(\grT,\Cloc)$ and $\CharVar{\grT}(\Cloc)$ defined by a certain analyticity condition.
\begin{definition}\label{def:differential-condition}
Let $W\subset  \Hom_{\Zp}(\grT,\Lloc)=\Hom_{\Qp}(\frt,\Lloc)$    
be a $\Lloc$-subvector space and define
\begin{equation*}
    C^{\an}_W(\grT,\Cloc):=\left\{f \in C^{\an}(\grT,\Cloc)\mid df\in  C^{\an}(\grT,\Cloc)\otimes_\Lloc W\right\},
\end{equation*}
and
\begin{equation*}
    \CharVar{\grT}_W(\Cloc):=\CharVar{\grT}(\Cloc) \cap C_W^{\an}(\grT,\Cloc).
\end{equation*}
We call $C_W^{\an}(\grT,\Cloc)$ the space of $W$-analytic functions and $\CharVar{\grT}_W(\Cloc)$ the \emph{$W$-analytic character group}. A subspace $W\subset  \Hom_{\Zp}(\grT,\Lloc)=\Hom_{\Qp}(\frt,\Lloc)$ as above is called an \emph{analyticity condition} on $T$ over $F$.
\end{definition}
From the formula $d(fg)=fdg+gdf$ it follows that $C_W^{\an}(\grT,\Cloc)$ is in fact an $\Cloc$-algebra.
\begin{remark} One has $C_W^{\an}(T,\Cloc)=C_{W_\Cloc}^{\an}(T,\Cloc)$ for $W_\Cloc:=W\otimes_\Lloc \Cloc$. The distinction between $\Lloc$ and $\Cloc$ will be important later, as we will see that $\CharVar{\grT}_W(\Cloc)$ are the closed points in $\Cloc$ of a rigid analytic variety defined over $\Lloc$. 
\end{remark}

\begin{example}\label{exmpl:DistrL}(Schneider-Teitelbaum \cite{ST})
Let $\Q_p\subset  \Lloc\subset \Cloc$ with $\Lloc$ finite over $\Qp$. The valuation ring $\grT:=\sO_\Lloc$ is a free $\Zp$-module of finite rank with Lie algebra $\frt=\Lloc$. We consider the one-dimensional subspace
$W:=\Hom_{\sO_\Lloc}(\sO_\Lloc,\Lloc)\subset\Hom_{\Zp}(\grT,\Lloc)$.
In this case, we have
\begin{align*}
	C_W^{\an}(\grT,\Cloc)&= C^{\Llocan}(\sO_\Lloc,\Cloc)\\
	\CharVar{(\sO_\Lloc)}_{W}(\Cloc)&= \CharVar{(\sO_\Lloc)}^{\Llocan}(\Cloc),
\end{align*}
where  $C^{\Llocan}(\sO_\Lloc,\Cloc)$ denotes the space of all locally $\Lloc$-analytic functions on $\sO_\Lloc$ and $\CharVar{(\sO_\Lloc)}^{\Llocan}(\Cloc)$ is the group of locally $\Lloc$-analytic characters studied in \cite{ST}.
\end{example}

\section{$W$-analytic character varieties}
In this section we show that the character group $\CharVar{\grT}_W(\Cloc)$ is the set of closed points  of a rigid analytic variety over $\Lloc$. Note that classical rigid analytic varieties can be seen as a full subcategory of adic spaces.

As a preparation, we show that for characters it suffices to check the analyticity condition at $0\in \grT$:
\begin{proposition}\label{lem:char-diagram} The group $\CharVar{\grT}_W(\Cloc)$ sits in a cartesian diagram 
\begin{equation*}
    \xymatrix{\CharVar{\grT}_W(\Cloc)\ar[r]\ar[d]&\CharVar{\grT}(\Cloc)\ar[d]^{(d-)|_0}\\
W_C\ar[r]& \Hom_{\Zp}(\grT,\Cloc),}
\end{equation*}
where the right vertical arrow is the map $\chi \mapsto (d\chi)|_0$, which assigns to $\chi$ the value of $d\chi$ at $0\in \grT$.
\end{proposition}
\begin{proof}

For $g\in \grT$, we have
\[
	(d\chi)|_g=\chi(g) \cdot (d\chi)|_0,
\]
hence $d\chi \in C^{\an}(\grT,\Cloc)\otimes_\Cloc W_C$ if and only if $(d\chi)|_0\in W_C$. This shows
\[
	\CharVar{\grT}_W(\Cloc)=\{ \chi \in \CharVar{\grT}(\Cloc) \mid (d\chi)|_0\in W_C\},
\]
and the Lemma follows.
\end{proof}
The group $\CharVar{\grT}(\Cloc)$ of locally $\Qp$-analytic characters is the group of $\Cloc$-valued points of a rigid analytic variety. The following theorem shows that this result generalizes to $W$-analytic character groups:
\begin{theorem}\label{thm:AdicCharVariety}
Let $\grT$ be a free $\Zp$-module of finite rank and $W\subset \Hom_{\Zp}(\grT, \Lloc)$ be a $\Lloc$-subvector space. Then the character group $\CharVar{\grT}_W(\Cloc)$ is the set of $\Cloc$-valued points of a rigid analytic space $\CharVar{\grT}_W$ over $\Lloc$. The rigid analytic space $\CharVar{\grT}_W$ fits into a cartesian diagram
\begin{equation}\label{eq:FiberProd}
    \xymatrix{\CharVar{\grT}_W\ar[r]\ar[d]&	\Hom_{\Zp}(\grT,\Zp)\otimes_{\Zp} (\Gmf)_\eta^{\rig}\ar[d]^{\log}\\
W\otimes_\Lloc \bbA^1 \ar[r]& \Hom_{\Zp}(\grT,\Lloc)\otimes_{\Lloc}\bbA^1=\Hom_{\Zp}(\grT,\Zp)\otimes_{\Zp}\bbA^1,}
\end{equation}
where $\bbA^1$ denotes the rigid analytic affine space over $\Lloc$ and $(\Gmf)_\eta^{\rig}$ denotes the rigid analytic generic fiber of the formal multiplicative group $\Gmf$ over $\Lloc$.
\end{theorem}
\begin{proof}
Let us define $\CharVar{\grT}_W$ as the rigid analytic space defined by the fiber product in  \eqref{eq:FiberProd}. We have to show that the group of $\Cloc$-valued points of  $\CharVar{\grT}_W$ is exactly the $W$-analytic character group.

 Note that the $\Cloc$-valued points of $\Hom_{\Zp}(\grT,\Zp)\otimes_{\Zp} (\Gmf)_\eta^{\rig}$ are exactly the locally $\Qp$-analytic characters of $\grT$. More precisely, the group of $\Cloc$-valued points of $(\Gmf)_\eta^{\rig}$ is given by
 \[
 	(\Gmf)_\eta^{\rig}(\Cloc)=\{z\in \Cloc:|z-1|<1\},
 \]
 and we get an isomorphism
 \begin{equation}\label{eq:QpCharIso}
 	\Hom_{\Zp}(\grT,\Zp)\otimes_{\Zp} (\Gmf)_\eta^{\rig}(\Cloc)\xrightarrow{\sim} \CharVar{\grT}(\Cloc), \quad \beta\otimes z \mapsto (g\mapsto \chi_{\beta\otimes z}(g)),
 \end{equation}
 with $\chi_{\beta\otimes z}(g):=z^{\beta(g)}=\exp(\beta(g)\cdot \log z)$.
 The formula 
 \[
 	(d\chi_{\beta\otimes z})|_0=\beta \cdot \log z \in \Hom_{\Zp}(\grT,\Cloc)
 \]
 shows that the isomorphism \eqref{eq:QpCharIso} fits into a commutative diagram
\begin{equation*}
    \xymatrix{\Hom_{\Zp}(\grT,\Zp)\otimes_{\Zp} (\Gmf)_\eta^{\rig}(\Cloc)\ar[r]^-{\sim}\ar[d]^{\id\otimes \log}&	\CharVar{\grT}(\Cloc)\ar[dl]^-{(d-)|_0}\\
\Hom_{\Zp}(\grT,\Zp)\otimes_{\Zp} \bbA^1 .}
\end{equation*}
  This shows that the $\Cloc$-valued points of the lower-left part of the diagram \eqref{eq:FiberProd} is isomorphic to:
\begin{equation*}
    \xymatrix{&	\CharVar{\grT}(\Cloc) \ar[d]^{(d-)|_0}\\
W_C \ar[r]& \Hom_{\Zp}(\grT,\Cloc),}
\end{equation*}
and the theorem follows from Proposition \ref{lem:char-diagram}.
\end{proof}
\begin{definition} \label{def:W-analytic-character-variety}
For a finite free $\Zp$-module $\grT$ and $W\subseteq\Hom_{\Zp}(\grT,\Lloc)$, we call the rigid analytic space $\CharVar{\grT}_W$ defined over $\Lloc$ the
\emph{$W$-analytic character variety} and we denote by $\CatCharVar_\Lloc$ the full subcategory of the category of rigid analytic group varieties over $\Lloc$ defined by these objects. 
\end{definition}

Note that we have the following functoriality: Let $T_1$ and $T_2$ be finite free $\Zp$-modules together with $\Lloc$-subvector spaces  $W_1\subseteq \Hom_{\Zp}(T_1,\Lloc)$ and $W_2\subseteq \Hom_{\Zp}(T_2,\Lloc)$. Any morphism $\varphi\colon T_2\to T_1$ such that $\varphi^*\colon \Hom_{\Zp}(T_1,\Lloc)\to  \Hom_{\Zp}(T_2,\Lloc)$ satisfies $\varphi^*(W_1)\subseteq W_2$ induces a morphism in $\CatCharVar_\Lloc$:
\[
	\CharVar{(\grT_1)}_{W_1}\to \CharVar{(\grT_2)}_{W_2}.
\]

\section{The Fourier transform} Let $\grT$ be a free $\Zp$-module of finite rank as before and $\frt:=\grT\otimes_{\Zp}\Qp$ be the Lie algebra of $\grT$ as a $\Qp$-analytic group. The algebra of $\Cloc$-valued distributions $D(\grT,\Cloc)$ is the dual of the topological $C$-algebra $C^{\an}(\grT,\Cloc)$ equipped with the strong dual topology. This is a commutative Fr\'echet $\Cloc$-algebra, where multiplication is given by the convolution product $\ast$.  

The \emph{Fourier transform} of $\lambda \in D(\grT,\Cloc)$ is the map
\[
	\cF_\lambda\colon \CharVar{\grT}(\Cp)\to \Cp,\quad \chi\mapsto \lambda(\chi).
\]
It is not difficult to see that $\cF_{\lambda\ast \mu}=\cF_\lambda \cdot \cF_\mu$ for all $\lambda,\mu \in D(\grT,\Cloc)$ (see \cite[Proposition 1.4]{ST}). The main result of $p$-adic Fourier theory for $\Qp$-analytic function is:
\begin{theorem}[Amice]\label{thm:Amice}
The Fourier transform induces an isomorphism of Fr\'echet algebras
\[	
	\cF\colon D(\grT,\Cloc)\xrightarrow{\sim} \cO(\CharVar{\grT}/\Cloc),
\]
where $\cO(\CharVar{\grT}/\Cloc)$ denotes the algebra of rigid analytic functions on the character variety $\CharVar{\grT}$.
\end{theorem}
\begin{proof}
See \cite[Theorem 1.3]{Amice78} or \cite[Theorem 2.2]{ST}.
\end{proof}

The goal of this subsection is to generalize the theorem by Amice to $W$-analytic distributions. As before, let $W\subseteq \Hom_{\Zp}(\grT,\Lloc)$ be an $\Lloc$-subvector space.
\begin{definition}
	The topological $\Cloc$-algebra of $\Cloc$-valued $W$-analytic distributions is the strong continuous dual $D_W(\grT,\Cloc):=C^{\an}_W(\grT,\Cloc)'$ of the topological $\Cloc$-algebra $C^{\an}_W(\grT,\Cloc)$. The product is given by the convolution of distributions.
\end{definition}
The following is a generalization of \cite[Proposition 1.4]{ST} to $W$-analytic distributions:
\begin{proposition}
For $\lambda \in D_W(\grT,\Cloc)$, we have $\lambda=0$ if and only if $\cF_\lambda|_{\CharVar{\grT}_W(\Cloc)}=0$.
\end{proposition}
\begin{proof}
We follow the proof of \cite[Proposition 1.4]{ST}. For any open subset $U\subseteq \grT$ and any function $f$ on $\grT$, let us write $f|U$ for the function which is the extension of $f|_U$ by zero. Any element $h\in \Sym^\bullet_\Cloc (W\otimes_\Lloc\Cloc)$ can be seen as a locally analytic function $h \colon \grT\to \Cloc$. In fact, by the definition of $C_W^{\an}(\grT,\Cloc)$, the set of functions $h|U$ where $U$ runs through open subsets of $\grT$ and $h$ through $\Sym^\bullet_\Cloc( W\otimes_\Lloc\Cloc)$ is dense in $C_W^{\an}(\grT,\Cloc)$.
Suppose now that $\cF_\lambda|_{\CharVar{\grT}_W(\Cloc)}=0$, i.e. that $\lambda(\chi)=0$ for all $\chi\in \CharVar{\grT}_W(\Cloc)$. Using the character theory of finite abelian groups, we conclude
\[
	\lambda(\chi|U)=0, \quad \text{ for all open $U\subseteq \grT$ and all } \chi \in  \CharVar{\grT}_W(\Cloc).
\]
For any $\beta\in W$ and $z\in \Cloc$ of sufficiently small norm, we consider the character $\chi_{\beta\otimes z}\in \CharVar{\grT}_W(\Cloc)$, and obtain by continuity for $U\subseteq \grT$ open:
\[
	0=\lambda(\chi_{\beta\otimes z}|U)=\sum_{n\geq 0}\frac{z^n}{n!}\lambda(\beta^n|U).
\]
Since the right hand side is analytic in $z$ in a neighbourhood of $0\in \Cloc$, we deduce $\lambda(\beta^n|U)=0$ for all $\beta\in W$, $n\geq 0$ and $U\subseteq \grT$ open. Now note that over a field   $k$ of characteristic zero and a finite dimensional $k$-vector space $V$ any $\Sym^\bullet_k V$ is spanned over $k$ by the set $\beta^n$ with $\beta \in V$ and $n\geq 0$, see for example \cite[Theorem 1]{Schinzel}. Applying this to $k=\Cloc$ and $V=W\otimes_\Lloc\Cloc$, we deduce that $\lambda(h|U)=0$ for any $h\in \Sym^\bullet_\Cloc( W\otimes_\Lloc\Cloc)$ and any $U\subseteq \grT$ open. Since these functions are dense in $C_W^{\an}(\grT,\Cloc)$, we deduce $\lambda=0$.
\end{proof}
As an immediate consequence one gets:
\begin{corollary}\label{cor:Ideal-IW}
	Let $W\subseteq \Hom_{\Zp}(\grT,\Lloc)$ be an $\Lloc$-subvector space. We define
	\[
		I_W:=\ker( D(\grT,\Cloc)\to D_W(\grT,\Cloc) ),
	\]
	where the surjection is induced by the inclusion $C^{\an}_W(\grT,\Cloc)\subseteq C^{\an}(\grT,\Cloc)$. Then
	\[
		I_W=\{\lambda\in D(\grT,\Cloc): \cF_\lambda|_{\CharVar{\grT}_W(\Cloc)}=0\}.
	\]
\end{corollary}

As we have identified the Lie algebra $\frt$ of $\grT$ with $\grT\otimes_{\Zp}\Qp$, we have an action of $\frt=\grT\otimes_{\Zp}\Qp$ on $C^{\an}(\grT,\Cloc)$  in the usual way: for $\Xi\in \frt$ we let
\begin{equation*}
    \Xi:C^{\an}(\grT,\Cloc)\to C^{\an}(\grT,\Cloc)
\end{equation*}
be the composition of 
\begin{equation*}
    	d\colon C^{\an}(\grT,\Cloc)\to C^{\an}(\grT,\Cloc)\otimes_\Cloc \Hom_{\Zp}(\grT,\Cloc)
\end{equation*}
with the evaluation map 
\begin{equation*}
    \Hom_{\Zp}(\grT,\Cloc)\times \frt\to \Cloc. 
\end{equation*}
The resulting action of $\Xi\in \frt$ on $f\in C^{\an}(\grT,\Cloc)$ is denoted by $\Xi.f$. This action allows us to define a map
\[
	\iota\colon \frt \to D(\sG,\Cloc),\quad \Xi\mapsto (f\mapsto (\Xi.f)(0)).
\]
\begin{proposition}\label{lem:eq-for-charW}
	Let $W\subseteq \Hom_{\Zp}(\grT,\Lloc)=\Hom_{\Qp}(\frt,\Lloc)$ be a $\Lloc$-subvector space, and consider the orthogonal complement of $W$ in $\frt$:
	 \[
	 W^\perp:=\{ \Xi\in \frt: w(\Xi)=0 \text{ for all } w\in W  \}\subseteq \frt
	 \] 
	 We have
	 \[
		\CharVar{\grT}_W(\Cloc)=\{\chi\in \CharVar{\grT}(\Cloc):\cF_{\iota(\Xi)}(\chi)=0 \text{ for all } \Xi\in W^\perp\}.
	\]
\end{proposition}
\begin{proof}
This follows from the following equivalences for a character $\chi\in \CharVar{\grT}$:
\begin{align*}
	\cF_{\iota(\Xi)}(\chi)=0 \quad \forall \Xi\in W^\perp\quad
	&\Leftrightarrow\quad (\Xi.\chi)(0)=0, \quad \forall \Xi\in W^\perp\\
	&\Leftrightarrow\quad(d\chi)|_0\in W 
	\quad\Leftrightarrow\quad \chi \in \CharVar{\grT}_W(\Cloc),
\end{align*}
where the last equivalence follows from Proposition \ref{lem:char-diagram}.
\end{proof}

\begin{remark}
We already know from Theorem \ref{thm:AdicCharVariety} that $\CharVar{\grT}_W$ is a closed reduced rigid analytic subspace of the character variety $\CharVar{\grT}$. Let us choose an $\Lloc$-basis $\Xi_1,\dots,\Xi_r$ of $W^\perp$. Proposition \ref{lem:eq-for-charW} tells us more explicitly that $\CharVar{\grT}_W$ is cut out by the equations 
 \[
 \cF_{\iota(\Xi_1)}=0,\dots, \cF_{\iota(\Xi_r)}=0.
 \]
\end{remark}

One  can now prove a generalization of the theorem of Amice for $W$-analytic distributions:
\begin{theorem}\label{thm:pAdicFourierMainTheoremW}
	The Fourier transform induces an isomorphism of topological $C$-algebras
	\[
		\cF_W\colon D_W(\grT,\Cloc) \xrightarrow{\sim} \cO(\CharVar{\grT}_W/\Cloc), \quad \lambda\mapsto \cF_\lambda|_{\CharVar{\grT}_W},
	\]
	that fits into a commutative diagram
	\[
		\xymatrix{
			\cF\colon D(\grT,\Cloc) \ar[r]^{\cong}\ar[d] & \cO(\CharVar{\grT}/\Cloc)\ar[d] \\
			\cF_W\colon D_W(\grT,\Cloc) \ar[r]^{\cong} & \cO(\CharVar{\grT}_W/\Cloc).
		}
	\]
	Here, the left vertical map is induced by the inclusion $C_W^{\an}(\grT,\Cloc)\subseteq C^{\an}(\grT,\Cloc)$ and the right vertical map is induced by the closed immersion $\CharVar{\grT}_W\subseteq \CharVar{\grT}$.
\end{theorem}
\begin{proof}
By \cite[9.5.2, Cor. 6]{BGR} and \cite[9.5.3, Prop. 4]{BGR}, the coherent sheaf $\cI_W\subseteq \cO_{\CharVar{\grT}}$ corresponding to the closed reduced subvariety $\CharVar{\grT}_W\subseteq \CharVar{\grT}$ consists of all local sections vanishing on the subvariety $\CharVar{\grT}_W\subseteq \CharVar{\grT}$. By Theorem \ref{thm:Amice} (the Theorem of Amice) and Corollary \ref{cor:Ideal-IW}, we have $\cI_W(\CharVar{\grT})=\cF(I_W)$. Since $\CharVar{\grT}$ is a Stein space, the global section functor is exact on coherent sheaves and we get
\[
	D_W(\grT,\Cloc)=D(\grT,\Cloc)/I_W\xrightarrow{\sim}\cO(\CharVar{\grT}/\Cloc)/\cI_W(\CharVar{\grT})\cong \cO(\CharVar{\grT}_W/C).\qedhere
\]
\end{proof}

The main theorem of Schneider-Teitelbaum on locally $\Lloc$-analytic $p$-adic Fourier theory is obtained as the special case with $\grT=\sO_\Lloc$ and $W=\Hom_{\Lloc}(\Lloc,\Cloc)\subseteq \Hom_{\Qp}(\Lloc,\Cloc)$:

\begin{corollary}[{\cite[Theorem 2.3]{ST}}]
Consider a finite field extension $\Lloc$ of $\Qp$, and set $T:=\sO_\Lloc$ and $W:=\Hom_{\Lloc}(\Lloc,\Lloc)\subseteq \Hom_{\Qp}(\Lloc,\Lloc)$. In this case, the isomorphism $F_W$ coincides with the isomorphism in \cite[Theorem 2.3]{ST}:
\[
	 D^{\Llocan}(\sO_\Lloc,\Cloc)\xrightarrow{\sim} \cO(\CharVar{(\sO_\Lloc)}^{\Llocan}/\Cloc).
\]
between locally $\Lloc$-analytic distributions and the corresponding character variety.
\end{corollary}
\begin{proof}
This follows from Theorem \ref{thm:pAdicFourierMainTheoremW} and Example \ref{exmpl:DistrL}.
\end{proof}

\begin{definition}For $\mu\in D_W(\grT,\Cloc)$ and $h\in C^{\an}_W(\grT,\Cloc)$ we write as usual
\begin{equation*}
    \int_Thd\mu:=\mu(h).
\end{equation*}
For $f\in \cO(\CharVar{\grT}_W/\Cloc)$, we denote by $\mu_f\in D_W(\grT,\Cloc)$ the corresponding distribution with $\cF_W(\mu_f)=f$, so that by definition
\begin{equation*}
    \int_T\chi d\mu_f=f(\chi)
\end{equation*}
for $\chi\in  \CharVar{\grT}_W(\Cloc)$.
\end{definition}
By Theorem \ref{thm:AdicCharVariety}, the Lie algebra of the rigid analytic variety $\CharVar{\grT}_W/\Cloc$ is given by $W_\Cloc:=W\otimes_\Lloc \Cloc$, and $W_\Cloc$ acts by derivations on $\cO(\CharVar{\grT}_W/\Cloc)$.
On the other hand, the elements in $W_\Cloc$ are linear forms $T\to C$, i.e. elements of $C^{\an}_W(\grT,\Cloc)$. This gives rise to two interpretations of $\Xi\in \Sym_\Cloc^{\bullet}W_C$:

\begin{notation}\label{not:diff-op}
We see $\Xi\in \Sym_\Cloc^{\bullet}W_C$
\begin{enumerate}
\item as an invariant differential operator $\Xi:\cO(\CharVar{\grT}_W/\Cloc) \to \cO(\CharVar{\grT}_W/\Cloc), f\mapsto \Xi.f$,
\item as a polynomial function $P_\Xi\in C^{\an}_W(\grT,\Cloc)$, so that $f\mapsto P_\Xi f$ is the multiplication by  $P_\Xi$.
\end{enumerate}
For a distribution $\lambda\in D_W(\grT,\Cloc)$ we define $P_\Xi\lambda$ by
\begin{equation*}
    P_\Xi\lambda(f):= \lambda(P_\Xi f).
\end{equation*}
\end{notation}

These two interpretations are related by the following integration formula:
\begin{proposition}\label{prop:dist-integration-formula}
	For $f\in \cO(\CharVar{\grT}_W/\Cloc)$ and $\Xi\in \Sym_\Cloc^\bullet W_\Cloc$, we have:
	\[
		P_\Xi\mu_f=\mu_{\Xi.f},
	\]
	or stated differently: For every $\chi \in \CharVar{\grT}_W(\Cloc)$, we have:
	\begin{equation*}
	\int_{\grT} \chi\cdot P_\Xi\,d\mu_f=(\Xi.f)(\chi).
	\end{equation*}
\end{proposition}
\begin{proof}
In the case $W=\Hom_\Zp(\grT,\Lloc)$, the Fourier transform $F_W$ is just the classical Amice transform $F$ and the statement is well-known, see e.g. \cite[Lemma 4.6 (8)]{ST}. For general $W\subseteq \Hom_{\Zp}(\grT,\Cloc)$, the statement follows from Theorem \ref{thm:pAdicFourierMainTheoremW} and the commutativity of
\[
	\begin{gathered}[b]
		\xymatrix{
			I_W 		\ar[r]^{P_\Xi\cdot }\ar[d]^{\cF} & I_W \ar[d]^{\cF} \\
		 	\cF(I_W) 	\ar[r]^{\Xi.} & \cF(I_W) . 
		}\\[-\dp\strutbox]
	\end{gathered}\qedhere
\]
\end{proof}

The locally constant function are in $C_W^{an}(T,C)$ and we will need a formula for their integral. Suppose now that $C$ contains all $p$-power roots of unity and fix an embedding of $\QQ(\mu_{p^\infty})\hookrightarrow C$. Any $p^n$-torsion point $\chi \in \widehat{T}_W(C)[p^n]$ is a character of the form
\[
	\chi\colon T\to T/p^nT\to \mu_{p^n}(C)\subseteq C^\times.
\]
\begin{definition}\label{def:Fourier-transform}
	Let $R$ be a $\QQ(\mu_{p^\infty})$-algebra. For any locally constant function $\rho \colon T\to R$, the Fourier transform $\widehat{\rho}$ of $\rho$ is a finitely supported function $\widehat{\rho}\colon \bigcup_n \widehat{T}_W(C)[p^n]\to R$ defined as follows: Let $n$ be an integer such that $\rho$ factors as $\rho \colon T\to T/p^nT\to R$, then $\widehat{\rho}\in C(\widehat{T}_W(C)[p^n],R)$ is given by
	\[
		\widehat{\rho}(\chi):=\frac{1}{\#(T/p^nT)}\sum_{\widetilde{s}\in T/p^nT} \chi(\widetilde{s})^{-1}\rho(\widetilde{s}).
	\]
\end{definition}
Note that the Fourier transform $\widehat{\chi}$, for $\chi\in \widehat{T}_W(C)[p^n]$, is the characteristic function of $\{\chi\}\subseteq \bigcup_n \widehat{T}_W(C)[p^n]$.
\begin{corollary}\label{cor:dist-integration-formula}
Let $f\in \cO(\widehat{T}_W/C)$.
For any locally constant function $\rho \colon T\to R$ and any $\Xi\in \Sym^\bullet_C W_C$, we have
\[
	\int_T \rho\cdot P_\Xi d\mu_f=\sum_{\chi\in \bigcup_n \widehat{T}_W(C)[p^n]} \widehat{\rho}(\chi)\cdot  (\Xi.f)(\chi).
\]
\end{corollary}
\begin{proof}
Since any locally constant function can be written as a finite sum of characters of the form $\chi$ with $\chi\in \widehat{T}_W(C)[p^n]$, it suffices to prove the formula for $\rho=\chi$. In this case, we have according to Proposition \ref{prop:dist-integration-formula}:
\[
	\int_T \rho\cdot P_\Xi d\mu_f=\int_T \chi\cdot P_\Xi d\mu_f=(\Xi.f)(\chi)=(\Xi.f)(\chi),
\]
This proves the formula in this case, since $\widehat{\chi}$ is the characteristic function of $\{\chi\}$.
\end{proof}

\section{$W$-analytic character varieties for abelian $p$-adic Lie groups}\label{sec:p-adic-Lie}
The above $p$-adic Fourier theory for $T$ generalize easily to  arbitrary topologically finitely generated abelian $p$-adic Lie groups. We will need this generalization later in our applications to the interpolation of $L$-values.

Let $A$ be a topologically finitely generated abelian $p$-adic Lie group and
$\fra$ be its Lie algebra.  We fix an exponential map $\exp_A\colon \fra_0\to A$ defined on a compact open $\Zp$-module $\fra_0\subseteq \fra$.

\begin{proposition}[{\cite[Prop. 6.4.5]{Emerton2017}}]
There is a strictly quasi-Stein rigid analytic variety $\widehat{A}$ over $\Qp$ parametrizing locally $\Qp$-analytic characters. 
\end{proposition}
\begin{definition}
We call $\widehat{A}$ the ($\Qp$-analytic) character variety of $A$. 
\end{definition}

For a complete non-Archimedean extension $C$ of $\Qp$, we write $C^{\an}(A,C)$ for the locally $\Qp$-analytic functions on $A$ with values in $C$ and $D(A,C):=C^{\an}(A,C)'$ for the locally $\Qp$-analytic distributions. 

\begin{proposition}[{\cite[Prop. 6.4.6]{Emerton2017}}]\label{prop:Emerton-Fourier}
There is an injective continuous map of topological $C$-algebras
\[
	D(A,C)\hookrightarrow \cO(\widehat{A}/C),
\]
which is an isomorphism if $A$ is compact.
\end{proposition}

To define the analyticity condition, let $W\subseteq \Hom_{\Qp}(\fra, F)$ be an $F$-subvector space defined over some complete non-Archimedean extension of $\Qp$. This defines an analyticity condition on the free  $\Zp$-module $\fra_0$, see Definition \ref{def:differential-condition}, and hence we have the $W$-analytic character variety $\widehat{(\fra_0)}_W$, see Definition \ref{def:W-analytic-character-variety}. 

The exponential map defines by duality a homomorphism of rigid analytic groups $\widehat{A}\to \widehat{\fra}_0$. We use this to define the $W$-analytic character variety of $A$:

\begin{lemma}
Define the $W$-analytic character variety $\widehat{A}_W$ as the pullback
\[
	\xymatrix{
		\widehat{A}_W \ar[r]\ar[d] & \widehat{A}\ar[d]\\
		\widehat{(\fra_0)}_W \ar[r] & \widehat{\fra}_0.
	}
\]
Then this definition does not depend on the choice of the compact open subgroup $\fra_0\subseteq \fra$.
\end{lemma}
\begin{proof}
Suppose one has another compact open subgroup $\fra_0'$ on which the exponential map is defined. Without loss of generality, one may assume $\fra_0\subseteq \fra'_0$. It follows from the definition of $\widehat{\fra}_0$ that one has  a cartesian diagram
\[
	\xymatrix{
		\widehat{(\fra'_0)}_W \ar[r]\ar[d] & \widehat{\fra}_0'\ar[d]\\
		\widehat{(\fra_0)}_W \ar[r] & \widehat{\fra}_0,
	}
\]
and the independence of $\fra_0\subseteq \fra$ follows.
\end{proof}

For any complete non-Archimedean field $C$ containing $F$, we define the locally $W$-analytic functions as above to be
\begin{equation}\label{eq:analytic-functions-on-A}
    	C^{\an}_W(A,C):=\{f\in C^{\an}(A,C): df \in C^{\an}(A,C)\otimes_F W\}
\end{equation}
and the $W$-analytic distributions to be its dual
\[
	D_W(A,C):=C_W(A,C)'.
\]
Then $\widehat{A}_W(C)$ consists exactly of the locally analytic characters $\chi \colon A\to C^\times$ such that $\chi \in C^{\an}_W(A,C)$. One can now easily generalize the main theorem of $p$-adic Fourier theory from Theorem \ref{thm:pAdicFourierMainTheoremW} to more general abelian $p$-adic Lie groups:

\begin{theorem}[$p$-adic Fourier theory for abelian $p$-adic Lie groups]\label{thm:FT-abelian}
	Let $A$ be a compact topologically finitely generated abelian $p$-adic Lie group and $W\subseteq \Hom_{\Qp}(\fra,F)$. The Fourier transform induces an isomorphism of topological $C$-algebras
	\[
		\cF_W\colon D_W(A,\Cloc) \xrightarrow{\sim} \cO(\CharVar{A}_W/\Cloc), \quad \lambda\mapsto \cF_\lambda|_{\CharVar{A}_W},
	\]
	that fits into a commutative diagram
	\[
		\xymatrix{
			\cF\colon D(A,\Cloc) \ar[r]^{\cong}\ar[d] & \cO(\CharVar{A}/\Cloc)\ar[d] \\
			\cF_W\colon D_W(A,\Cloc) \ar[r]^{\cong} & \cO(\CharVar{A}_W/\Cloc),
		}
	\]
	where the upper horizontal isomorphism comes from Proposition \ref{prop:Emerton-Fourier}.
\end{theorem}
\begin{proof}
The proof is the same as in Theorem \ref{thm:pAdicFourierMainTheoremW} using that $\CharVar{A}$ is strictly  quasi-Stein.

\end{proof}

Let us finally observe that a homomorphism of topologically finitely generated abelian $p$-adic Lie groups $\varphi\colon B \to A$ together with the analyticity condition $W\subseteq \Hom(\frb,F)$ on $B$, induces in a natural way an analyticity condition $W_\varphi\subseteq \Hom(\fra,F)$ through the pull-back diagram
\[
	\xymatrix{	W_\varphi \ar[r]\ar[d] & \Hom(\fra,F)\ar[d]^{\varphi^*}\\
	W\ar[r] & \Hom(\frb,F).
	}
\]
Usually, the map $\varphi\colon B\to A$ will be clear from the context. To avoid double subscripts, we will write 
\begin{align}\label{eq:induced-W}
	C_W^{\an}(A,C):= C_{W_\varphi}^{\an}(A,C), \quad D_{W}(A,C):=D_{W_\varphi}(A,C), \quad \text{and}\quad \widehat{A}_{W}:=\widehat{A}_{W_\varphi}.
\end{align}
By construction, one has 
\[
	C_W^{\an}(A,C)=\{f\in C^{\an}(A,C): f\circ \varphi \in C_W(B,C)\}.
\]

\begin{lemma}
Let $\varphi\colon B\to A$ be a homomorphism of $p$-adic Lie groups as above and let $W\subseteq \Hom(\frb,F)$ be an analyticity condition on $B$. Then one has  a cartesian diagram
\[
	\xymatrix{	\widehat{A}_W \ar[r]\ar[d] & \widehat{A}\ar[d]\\
	\widehat{B}_W \ar[r] & \widehat{B}.
	}
\]
\end{lemma}
\begin{proof}
This follows immediately from the definitions.
\end{proof}

\section{Abelian Galois groups and their character varieties}\label{sec:Hecke-character-variety}
In part II the Fourier transformation will be applied to an abelian Galois group, which gives rise to a  rigid analytic variety parametrizing all $p$-adic Hecke characters of conductor dividing $p^\infty \frf$. In this section, we discuss the relevant definitions and introduce the concept of $\Sigma$-analytic distributions for a subset $\Sigma$ of embeddings of a number field $L$.

Let $L$ be a number field and fix an embedding $\overline{\QQ}\hookrightarrow \Cp$. We write $\sO_L$ for the ring of integers in $L$ and $\bbA^\times_L$ for the group of ideles over $L$. 
For  an  ideal $\frf\subset \sO_L$ co-prime to $p$ we let
\begin{equation*}
    G(p^\infty\frf):=\Gal(L(p^\infty\frf)/L)
\end{equation*} 
be the Galois group of the ray class extension for $p^{\infty}\frf$. This is a compact and topologically finitely generated abelian $p$-adic Lie group. 
\begin{definition}
A \emph{$p$-adic Hecke character} is a continuous group homomorphism 
\[
	\chi \colon \bbA^\times_L/L^\times \to \CC_p^\times.
\]
Any such $p$-adic Hecke character factors through an abelian Galois group $G(p^\infty\frf)$ for some ideal $\frf\subset \sO_L$ co-prime to $p$
\[
	\chi \colon \bbA^\times_L/L^\times \twoheadrightarrow G(p^\infty\frf) \to \CC_p^\times.
\]
In this case we say that $\chi$ has conductor dividing $p^\infty \frf$ and we abuse notation by writing also $\chi: G(p^{\infty}\frf)\to \CC_p^\times$. 
\end{definition}
In particular, $p$-adic Hecke characters are exactly the $1$-dimensional $p$-adic Galois representations of the number field $L$.

\begin{definition}
The character variety $\widehat{G(p^\infty \frf)}$ associated to $G(p^\infty\frf)$ is called the \emph{$p$-adic Hecke character variety of conductor $p^\infty \frf$}.
\end{definition}

Any $p$-adic Hecke character is automatically locally $\Qp$-analytic. In particular, the $\Cp$-points of $\widehat{G(p^\infty \frf)}$ are exactly the $p$-adic Hecke characters of conductor dividing $p^\infty \frf$. 
\begin{remark}
The above point of view on  $p$-adic Hecke characters leads to consider $p$-adic $L$-functions as rigid analytic (or meromorphic) functions on the $p$-adic Hecke character variety. By $p$-adic Fourier theory, these functions are exactly the locally $\Qp$-analytic distributions on $G(p^\infty \frf)$. There is a unique injective map 
\[
\Zp\llbracket G(p^\infty \frf) \rrbracket \hookrightarrow D(G(p^\infty \frf),\Cp) \cong \cO(\widehat{G(p^\infty \frf)}/\Cp)
\]
of the Iwasawa algebra of $G(p^\infty\frf)$ to the rigid analytic functions on the $p$-adic Hecke character variety. Hence, this point of view aligns well with the measure-theoretic approach to $p$-adic $L$-functions. Needless to say that this point of view is motivated by Tate's Fourier theoretic approach to the functional equation of Hecke $L$-functions.
\end{remark}
From now on, let $C\subseteq \mathbb{C}_p$ complete fields containing the Galois closure of $L$ under the embedding $\overline{\QQ}\subseteq \Cp$. 

Given a set of embeddings $\Sigma\subseteq \Hom(L,\overline{\QQ})$, 
one has 
\[
	\Hom(\sO_L\otimes \Zp,\Cloc)\cong\prod_{\sigma \in \Hom(L,\overline{\QQ})} \Cloc
\]
so that for $T:=\sO_L\otimes \Zp$ we can define an analyticity condition, 
\begin{equation}\label{eq:W-Sigma}
	W(\Sigma):=\prod_{\sigma \in \Sigma} C\subseteq \prod_{\sigma \in \Hom(L,\overline{\QQ})} C=\Hom(\sO_L\otimes \Zp,C)=\Hom(\sO_L\otimes \Qp,C),
\end{equation}
which we call $\Sigma$-analyticity. 
\begin{definition}\label{def:SigmaAnalytic}
We say that $f\in C^{\an}(\sO_L\otimes \Zp,C)$ is \emph{locally $\Sigma$-analytic} if for every $x\in \sO_L\otimes \Zp$ there exists an open neighbourhood $U$ of $x$ such that $f|_U$ can be written as a convergent series
\[
	f|_U(z)=\sum_{\underline{n}\in \NN_0^\Sigma} a_{\underline{n}} z^{\underline{n}}, \quad a_{\underline{n}}\in C
\]
where $z^{\underline{n}}:=\prod_{\sigma\in \Sigma} \sigma(z)^{n_\sigma}$. 
\end{definition}
Similar functions have been studied before in \cite{DeIeso2013}. It is straightforward to check:
\begin{lemma}One has 
\begin{equation*}
    C^{\an}_{W(\Sigma)}(\sO_L\otimes \Zp,\Cloc)=\{f\in C^{\an}(\sO_L\otimes \Zp,C)\mid f \mbox{ locally $\Sigma$-analytic}\}.
\end{equation*}
\end{lemma}

We want to extend the analyticity condition $W(\Sigma)$ to $G(p^\infty\frf)$. For this we use the exact sequence
\begin{equation*}
    \xymatrix{0\ar[r] & \Gal(L(p^\infty\frf)/L(\frf)) \ar[r] &  \Gal(L(p^\infty\frf)/L) \ar[r]^{\pi} &  \Gal(L(\frf)/L)\ar[r] &  0
}
\end{equation*}
and the isomorphism
\[
 (\sO_L\otimes \Zp)^\times/E(\sO_\frf^\times) \xrightarrow{\sim}\Gal(L(p^\infty\frf)/L(\frf))\subseteq G(p^\infty\frf),
\]
coming from  class field theory,
where $E(\sO_\frf^\times)$ is the closure of   the one-units $\sO_\frf^\times$ modulo $\frf$  in the local units $(\sO_L\otimes \Zp)^\times$. In this way, we obtain a homomorphism 
\begin{equation*}
    \varphi\colon (\sO_L\otimes \Zp)^\times \to G(p^\infty\frf)
\end{equation*}
of $p$-adic Lie groups and we can apply the construction from \ref{eq:induced-W}. This gives an induced analyticity condition
\begin{equation*}
    W(\Sigma)_\varphi\subseteq \Hom(\Lie(G(p^\infty\frf)),C)
\end{equation*}
on $G(p^\infty \frf)$ and we write $\widehat{G(p^\infty \frf)}_{W(\Sigma)}$ for the associated $W(\Sigma)$-analytic character variety. As before, we drop the index $\varphi$ to avoid double subscripts.

\begin{definition}\label{def:Sigma-analytic-Galois}
	The rigid analytic variety $\widehat{G(p^\infty \frf)}_{W(\Sigma)}$ is called the \emph{$\Sigma$-analytic $p$-adic Hecke character variety}. The space from \eqref{eq:analytic-functions-on-A}
	\[
		C_{W(\Sigma)}(G(p^\infty \frf),C)\stackrel{\text{def.}}{=}\{ f\in C^{\an}(G(p^\infty \frf),C): df\in  C^{\an}(G(p^\infty \frf),C)\otimes W(\Sigma)_\varphi\}
	\]
is called \emph{the space of  $\Sigma$-analytic functions on the Galois group $G(p^\infty \frf)$}. Its strong dual space
\[
	D_{W(\Sigma)}(G(p^\infty \frf),C)\stackrel{\text{def.}}{=}C_{W(\Sigma)}(G(p^\infty \frf),C)'
\]
is called \emph{the space of  $\Sigma$-analytic distributions on $G(p^\infty \frf)$}.
\end{definition}

Spelling out the definitions, we can give the following more explicit description of the $\Sigma$-analytic functions on $G(p^\infty \frf)$. We consider the following diagram \[
\xymatrix{
	&(\sO_L\otimes \Zp)^\times \ar@{->>}[d]^{\mathrm{pr}} \ar[dr]^{\varphi}& & & \\
	0\ar[r] & \Gal(L(p^\infty\frf)/L(\frf)) \ar[r] &  \Gal(L(p^\infty\frf)/L) \ar[r]^{\pi} &  \Gal(L(\frf)/L)\ar[r] &  0.
}
\]
A locally $\Qp$-analytic function $f\in C^{\an}(G(p^\infty \frf),C)$ is contained in $C_{W(\Sigma)}(G(p^\infty \frf),C)$ if and only if for each $g\in \Gal(L(\frf)/L)$ the composition 
		\[
			\pr\circ (\rho|_{\pi^{-1}(\{g\})}) \colon (\sO_L\otimes \Zp)^\times \twoheadrightarrow \pi^{-1}(\{g\})\subseteq \Gal(L(p^\infty \frf)/L)\xrightarrow{\rho} \Cp
		\]
		is $\Sigma$-analytic in the sense of Definition \ref{def:SigmaAnalytic}. 

The following description of $\Sigma$-analytic distributions will be used later:
\begin{proposition}\label{prop:decomp-Gal-distributions}
One has  canonical isomorphisms of topological $C$-algebras
\begin{align*}
	D_{W(\Sigma)}(G(p^\infty \frf),C)&\cong \bigoplus_{g\in \Gal(L(\frf)/L)}D_{W(\Sigma)}(\Gal(L(p^\infty\frf)/L(\frf)),C)\\
	&\cong  \bigoplus_{g\in \Gal(L(\frf)/L)} D_{W(\Sigma)}((\sO_L\otimes \Zp)^\times/E(\sO_\frf^\times), C).
\end{align*}
\end{proposition}
\begin{proof}
Since $\frf$ is co-prime to $p$, one has an isomorphism
\[
	G(p^\infty \frf) \cong \Gal(L(p^\infty \frf))\times \Gal(L(\frf)/L),
\]
which gives the first isomorphism in the statement. The second isomorphism follows from  $\Gal(L(p^\infty\frf)/L(\frf)) \cong (\sO_L\otimes \Zp)^\times/E(\sO_\frf^\times)$.
\end{proof}

\begin{remark}
One can associate to any algebraic Hecke character
\[
	\chi \colon \bbA^\times_L/L^\times \to \ol{\QQ}^\times \subseteq \CC^\times
\]
of infinity type $\mu=(\mu_\sigma)_{\sigma}\in \ZZ^{\Hom(L,\ol{\QQ})}$ and conductor dividing $p^\infty \frf$ a $p$-adic Hecke character of conductor dividing $p^\infty \frf$. This will be discussed  in the spacial case of totally imaginary fields in section \ref{sec:p-adic-Hecke}. It is not difficult to see that the $p$-adic Hecke character associated to $\chi$ is locally $\Sigma$-analytic if and only if $\mathrm{supp}(\mu):=\{\sigma \in \Hom(L,\ol{\QQ}: \mu_\sigma \neq 0)\}$ is contained in $\Sigma$. For $\Sigma$ an induced CM type of a totally imaginary field, which contains a CM field, this will be explained in Lemma \ref{lem:Sigma-analytic-Hecke-char}.
\end{remark}

\section{$p$-divisible groups over $\sO_{\C_p}$} Here we review the classification of $p$-divisible groups over $\sO_{\C_p}$ by Scholze--Weinstein. Let us denote by $\pDiv_{\OCp}$ the category of $p$-divisible groups over $\OCp$. For a $p$-divisible group $G$ over $\OCp$, the  Hodge--Tate exact sequence for the dual $p$-divisible group $G^\vee$ gives an inclusion
\begin{equation}\label{eq:Hodge-Tate-inclusion}
	\Lie G\otimes_{\OCp}\Cp \subseteq T_p(G)\otimes_{\Zp}\Cp(-1).
\end{equation}
Scholze and Weinstein have classified $p$-divisible groups over $\OCp$ by such pairs $(\Lie G\otimes \Cp, T_p(G))$,  see \cite[Theorem 5.2.1]{ScholzeWeinstein}.

For our purposes, it is more convenient to express this classification  in terms of the Tate module $T_p(G^\vee)$  instead of $T_p(G)$. By Cartier duality, we have for any $G\in \mathrm{ob}(\pDiv_{\OCp})$ a canonical isomorphism
\[
	\Hom_{\Zp}(T_p(G^\vee),\Cp)\cong T_p(G)\otimes_{\Zp}\Cp(-1).
\] Therefore, we define:
\begin{definition}\label{def:dual-HT-pair}
We write $\HTDpairs$ for the following category:
\begin{itemize}
\item \textbf{Objects} given by pairs $(W,T)$ with $T$ a free $\Zp$-module of finite rank and $W\subseteq \Hom_{\Zp}(T,\Cp)$ a $\Cp$-subvector space.
\item \textbf{Morphisms}  $(W_1,T_1)\to (W_2,T_2)$  are $\Zp$-linear maps $\varphi \colon T_2\to T_1$ such that the induced map $\Hom_{\Zp}(T_1,\Cp)\to \Hom_{\Zp}(T_2,\Cp)$ maps $W_1$ into $W_2$.
\end{itemize}
\end{definition}
In terms of $T_p(G^\vee)$, \eqref{eq:Hodge-Tate-inclusion} reads
\[
	\Lie G\otimes_{\OCp}\Cp \subseteq \Hom_{\Zp}(T_p(G^\vee),\Cp),
\] 
and we obtain a functor
\[
	HT^\vee\colon \pDiv_{\OCp}\to \HTDpairs, \quad G\mapsto (\Lie G\otimes_{\OCp}\Cp ,T_p(G^\vee)).
\]

In terms of the functor $HT^\vee$, the classification result for $p$-divisible groups over $\OCp$ of
Scholze and Weinstein can be formulated as follows:

\begin{theorem}[Scholze--Weinstein]\label{thm:ScholzeWeinstein}
We have an equivalence of categories:
\begin{align*}
		HT^\vee \colon \pDiv_{\OCp}\xrightarrow{\sim} \HTDpairs, \quad G\mapsto (\Lie G\otimes \C_p, T_p(G^\vee)).
\end{align*}
\end{theorem}
\begin{proof}
This follows from \cite[Theorem 5.2.1]{ScholzeWeinstein}. Scholze and Weinstein formulate the classification of $p$-divisible groups in terms of pairs of the form $(W,T)$ with $T$ a free $\Zp$-module of finite rank and $W\subseteq T\otimes_{\Zp}\Cp(-1)$ a $\Cp$-subvector space. Their functor assigns $\Lie(G)\otimes \Cp, T_p(G)$ to a $p$-divisible group. However, by Cartier duality, we have $T_p(G)(-1)=\Hom(T_p(G^\vee),\Zp)$.
\end{proof}
For a full proof of the theorem we refer to \cite[Theorem 5.2.1]{ScholzeWeinstein}. We only briefly recall the following explicit description of a quasi-inverse
\[
	\HTDpairs\xrightarrow{\sim} \pDiv_{\OCp}, \quad (W,T)\mapsto G_{(W,T)}.
\]
Given $(W,T)\in \mathrm{ob}(\HTDpairs)$, the logarithm of the formal multiplicative group $\Gmf$ induces a map
\begin{equation}\label{eq:WtoLieG}
	\log\colon \Hom_{\Zp}(T,\Zp)\otimes_\Zp (\Gmf)^{\ad}_\eta \to \Hom_{\Zp}(T,\Cp)\otimes_{\Cp} \bbA^1,
\end{equation}
where $(\Gmf)^{\ad}_\eta$ is the (adic) generic fiber of the formal multiplicative group and $\bbA^1$ denotes the adic affine line over $\Cp$.
The following pullback diagram defines a $p$-divisible rigid-analytic group in the sense of \cite{Fargues19},
\begin{equation}\label{eq:ScholzeWeinsteinCartDiag}
	\xymatrix{
		(G_{(W,T)})_\eta^{\ad} \ar[r]\ar[d] &  \Hom_{\Zp}(T,\Zp)\otimes_\Zp (\Gmf)^{\ad}_\eta \ar[d]^-{\log}\\
		W\otimes_{\Cp} \mathbb{A}^1 \ar[r] & \Hom_{\Zp}(T,\Cp)\otimes_{\Cp} \bbA^1,
	}
\end{equation}
where the lower horizontal map is given by the inclusion $W\subseteq \Hom_{\Zp}(T,\Cp)$. In the proof of \cite[Theorem 5.2.1]{ScholzeWeinstein} it is finally shown that $(G_{(W,T)})_\eta^{\ad}$ is indeed the adic generic fiber of a $p$-divisible group $G_{(W,T)}$, which can be explicitly described as
\[
	G_{(W,T)}:=\coprod_{Y\subseteq (G_{(W,T)})_\eta^{\ad}} \Spf H^0(Y,\cO_Y^+).
\]

\section{Character varieties and $p$-divisible groups}\label{sec:p-div}
There is a canonical fully faithful functor from the category $\Rigid_\Cloc$ of $\Cloc$-rigid analytic varieties to the category of adic spaces over $C$
\[
	\Rigid_\Cloc\to \Adic_\Cloc, \quad X\mapsto X^{\ad}.
\]
For a $p$-divisible group $G$ over the valuation ring $\sO_\Cloc$ of $\Cloc$, let us write $G^{\rig}_\eta\in \Rigid_\Cloc$ for the rigid analytic generic fiber of $G$. The rigid analytic generic fiber functor is compatible with the adic generic fiber functor $(\cdot)_\eta^\ad$, i.e. one has an equivalence $(\cdot)^{\ad}\circ (\cdot)_\eta^{\rig}\simeq (\cdot)^{\ad}_\eta \colon \pDiv\to \Adic$.

Let $W\subset\Hom_{\Zp}(\grT,\Lloc)$ a $\Lloc$-subvector space. One obtains a pair $(W_{\Cp},T)\in \HTDpairs$ with $W_{\Cp}:=W\otimes_\Lloc \Cp$. 
Then $ \CharVar{\grT}_W$ is the $W$-analytic character variety of $T$, and $(G_{(W_{\Cp},T)})_\eta^{\rig}$ is the rigid generic fiber of the $p$-divisible group $G_{(W_{\Cp},T)}$ associated to the pair $(W_{\Cp},T)$, see \eqref{eq:ScholzeWeinsteinCartDiag}.

\begin{theorem}[Main theorem on $p$-adic Fourier theory of $p$-divisible groups]\label{thm:UniformizationOfCharVar}
There is an isomorphism of rigid analytic varieties over $\Cp$
\[
	(G_{(W_{\Cp},T)})_\eta^{\rig} \cong \CharVar{\grT}_W/\Cp,
\]
and the Fourier transformation induces an isomorphism of topological $\Cp$-algebras
	\[
		\cF_W\colon D_{W_\Cp}(\grT,\Cp) \xrightarrow{\sim} \cO((G_{(W_{\Cp},T)})_\eta^{\rig}).
	\]
\end{theorem}
\begin{proof}
By \eqref{eq:ScholzeWeinsteinCartDiag}, the adic space $(G_{(W_{\Cp},T)})_\eta^{\ad}$ is given by the Cartesian diagram
\begin{equation*}
	\xymatrix{
		(G_{(W_{\Cp},T)})_\eta^{\ad} \ar[r]\ar[d] &  \Hom_{\Zp}(\grT,\Zp)\otimes_\Zp (\Gmf)_\eta^{\ad} \ar[d]^-{\log}\\
		W_{\Cp}\otimes_{\Cp} \mathbb{A}^1 \ar[r] & \Hom_{\Zp}(\grT,\Zp)\otimes_{\Zp} \mathbb{A}^1.
	}
\end{equation*}
By equation \eqref{eq:FiberProd} in Theorem \ref{thm:AdicCharVariety}, the adic space associated to the character variety $\CharVar{\grT}_W/\Cp$ is given by the same pull-back diagram. The claim follows since the functor $(\cdot)^{\ad}$  commutes with fiber products.
\end{proof}
From now on, we see rigid analytic varieties as a full subcategory of adic spaces. The theorem implies that every generalized character variety is isomorphic to an essentially unique $p$-divisible group over $\Cp$, and conversely every $p$-divisible group over $\Cp$ gives rise to a character variety.

We discuss some examples of $p$-divisible groups over $\OCp$ and their associated character varieties. We start with the following trivial example:

\begin{corollary}
The pair $(\Hom_{\Zp}(T,\Cp),T)\in \HTDpairs$ corresponds to the $p$-divisible group $\Hom_{\Zp}(T,\Zp)\otimes \mu_{p^\infty}$ and one obtains the following isomorphism of rigid analytic varieties over $\Cp$
\[
	\CharVar{\grT}\cong \Hom_{\Zp}(T,\Zp)\otimes (\Gmf)_{\eta}^{\rig}.
\]
\end{corollary}
\begin{proof}
One has $HT^\vee(\Hom_{\Zp}(T,\Zp)\otimes \mu_{p^\infty})=(\Hom_{\Zp}(T,\Cp),T)$.
\end{proof}

As a less trivial example, one gets the main result of Schneider--Teitelbaum about the uniformization of locally $\Lloc$-analytic character varieties from Theorem \ref{thm:UniformizationOfCharVar}:

\begin{corollary}[{\cite[Theorem 3.6]{ST}}]
Let $\Lloc$ be a finite extension over $\Qp$. Let $G$ be the $p$-divisible group over $\OCp$ associated to a Lubin--Tate formal group with endomorphisms by $\sO_\Lloc$, then $HT^\vee(G)\cong  (\Hom_{\Lloc}(\Lloc,\Cp),\sO_\Lloc)\in \HTDpairs$. In particular, one has
\[
	\CharVar{T}^{\Llocan}\cong G^{\rig}_\eta
\] 
\end{corollary}
\begin{proof}
It suffices to show that $HT^\vee(G)\cong (\Hom_{\Lloc}(\Lloc,\Cp),\sO_\Lloc)\in \HTDpairs$. The image of the injective map in the Hodge-Tate sequence
\[
	\Lie G\otimes_{\OCp} \Cp \hookrightarrow T_pG(-1)\otimes_\Zp \Cp\cong \Hom_{\Zp}(T_p(G^\vee), \Cp).
\]
is exactly the sub-space of $\Hom_{\Zp}(T_pG^\vee, \Cp)$, where the endomorphisms $\sO_\Lloc$ of $F_{LT}$ act by the fixed inclusion $\sO_\Lloc\subseteq \Cp$. The Tate module $T_p(G^\vee)$ is a free $\sO_\Lloc$-module of rank $1$. Hence, after fixing a generator of $T_p(G^\vee)$ as $\sO_\Lloc$-module, one obtains the desired isomorphism in $\HTDpairs$:
 \[
 	HT^\vee(G)\cong (\Lie G\otimes_{\OCp} \Cp, T_p(G^\vee))  \cong (\Hom_{\Lloc}(\Lloc,\Cp),\sO_\Lloc).\qedhere
 \]
\end{proof}

\part{$p$-adic $L$-functions of Hecke characters}
We apply the results on the $p$-adic Fourier transform of the first part  to the interpolation of critical $L$-values of algebraic Hecke characters. We show that the generating function of these special values for Hecke characters of $\infty$-type $-\alpha$ constructed in \cite{Kings-Sprang} gives rise to rigid analytic functions on the $p$-divisible group of the attached abelian variety with complex multiplication, which  results  in the desired interpolation properties. 
\section{Hecke $L$-functions of totally imaginary fields}
We start with some notations, which will be used in the whole part II. Fix an algebraic closure $\ol{\QQ}$ of $\Q$ and once for all we fix an embedding
\begin{equation*}
    \iota_\infty:\ol{\QQ}\to \C.
\end{equation*}
Let $L\subset \ol{\QQ}$ be a totally imaginary field of degree $2d=[L:\Q]$ with integers $\sO_L$, discriminant $d_L$ and group of ideles $\bbA_L^\times$. We write $J_L:=\Hom_{\QQ}(L,\ol{\QQ})$ for the set of field embeddings of $L$ to $\ol{\QQ}$ and $\ZZ^{J_L}$ for the free abelian group on $J_L$. A \emph{Hecke character} is a continuous group homomorphism
\[
	\chi\colon \bbA_L^\times/L^\times \to \CC^\times.
\]
The largest ideal $\frf\subseteq \sO_L$ such that $\chi$ factors as
\[
	\chi\colon \bbA_L^\times/L^\times \twoheadrightarrow \bbA_L^\times/U(\frf)L^\times \to \CC^\times,
\]
where $U(\frf):=\ker\left( \sO_L\otimes\widehat{\Z}\to (\sO_L/\frf)^\times \right)$, is called the \emph{conductor} of $\chi$. A Hecke character is called \emph{algebraic of infinity type $\mu\in\Z^{J_L} $ } if its restriction to the infinite ideles $\bbA_{L,\infty}^\times$ is of the form
\[
	\chi|_{\bbA_{L,\infty}^\times}\colon \bbA_{L,\infty}^\times=(L\otimes \RR)^\times \to \bbA_L^\times/L^\times \xrightarrow{\chi} \CC^\times, \quad (x\otimes 1)\mapsto \prod_{\sigma\in J_K} \sigma(x)^{-\mu(\sigma)}.
\]
Let $\cI(\frf)$ be the group of fractional ideals coprime to $\frf$.
For any prime ideal $\frp\in \cI(\frf)$ we define
\[
	\chi(\frp):=
	\chi(\dots,1,\pi,1,\dots),
\]
where $\pi$ is a uniformizer of $\frp$  
and extend $\chi$ multiplicatively to   all  of $\cI(\frf)$. This allows us to give the following equivalent more classical point of view on Hecke characters. Let
\begin{equation*}
    \cP_{\frf}:=\{(\xi)\in \cI(\frf): \xi\equiv 1\,\mathrm{mod}^\times\, \frf \}\subset \cI(\frf),
\end{equation*}
be the subgroup of principle ideal $\equiv 1$ modulo $\frf$.
An algebraic Hecke character $\chi$  of infinity type $\mu\in\Z^{J_L} $ is a group homomorphism
\[
	\chi \colon \cI(\frf) \to \CC^\times,
\]
such that the restriction of $\chi$ to $\cP_\frf$ is given by 
\[
	\chi((\xi))=\xi^\mu:=\prod_\sigma \sigma(\xi)^{\mu(\sigma)}.
\]
The  Hecke $L$-function of $\chi$ modulo $\frf$ is
\[
	L_\frf(\chi,s):=\sum_{\substack{\fra\subseteq \sO_L \\ (\fra,\frf)=1}} \frac{\chi(\fra)}{N\fra^s}.
\]
If $\frf$ is the conductor of $\chi$, then we will simply write $L(\chi,s)$. In general, $L_\frf(\chi,s)$ differs from $L(\chi,s)$ exactly by the finitely many Euler factors corresponding to prime ideals dividing $\frf$. Let $\chi$ be an algebraic Hecke character of conductor $\frf$, then the completed Hecke $L$-function $\Lambda(\chi, s):=(|d_L|\cdot N\frf)^{s/2}\Gamma_\infty(\chi, s) \cdot L(\chi,s)$, where  $\Gamma_\infty(\chi,s)$ is essentially a product of $\Gamma$-functions, satisfies a functional equation
\[
	\Lambda(\chi,s )=\epsilon(\chi)\Lambda(\chi^{-1},1-s)
\]
for some complex number $\epsilon(\chi)$ which has absolute value $1$ under each embedding. Deligne calls $L(\chi,s)$ \emph{critical} at an integer $s=k$ if  neither $\Gamma_\infty(\chi, s)$ nor $\Gamma_\infty(\chi^{-1}, 1-s)$ has a pole at $s=k$. By multiplication with the norm character $N(\cdot)$ and by the formula $L(\chi\cdot N(\cdot)^{-k},s)=L(\chi, s+k)$, the investigation of special values of Hecke $L$-functions is reduced to the study of these functions at $s=0$. We call the Hecke character $\chi$ critical if $L(\chi, s)$ is critical at $s=0$. Critical Hecke characters $\chi$ of infinity type $\mu$ can only exist, if there exists a CM sub-field $K\subseteq L$ and an infinity type $\mu_K\in \ZZ^{J_{K}}$ such that $\mu$ is induced from $\mu_K$, i.e.,
\[
	\mu(\sigma)=\mu_K(\sigma|_K).
\]	
Furthermore, the set $\Sigma_0:=\{ \sigma_0\in J_K: \mu_K(\sigma_0)<0 \}$ is a CM type of the CM field $K$, i.e., $J_K$ is the disjoint union of $\Sigma_0$ and the complex conjugate set $\overline{\Sigma}_0$. Thus,  a critical Hecke character $\chi$ of a totally imaginary field $L$ does always come with a CM subfield $K\subseteq L$ and a CM type $\Sigma_0$. The subset 
\begin{equation}\label{eq:inducedCMtype}
	\Sigma:=\{ \sigma \in J_L: \sigma|_K \in \Sigma_0 \}
\end{equation}
is called \emph{the CM type induced by $\Sigma_0$}. If we write \emph{$\Sigma$ is the CM type of $L$} we will always mean that $\Sigma$ is a subset of $J_L$ which is induced by a CM type $\Sigma_0$ of a CM subfield, and we will write $\ol{\Sigma}$ for the CM type induced by $\ol{\Sigma}_0$.

\begin{notation}\label{not:CMfields}
We write $K$ for the maximal CM subfield of $L$ and we will fix a CM type $\Sigma\subseteq J_L$ of $L$, induced from $K$, and consider critical Hecke characters $\chi$ of this fixed CM type.
\end{notation} 
In particular, a general critical Hecke character of CM type $\Sigma$ has an infinity type of the form
\[
	\mu=\beta -\alpha
\]
with $\beta \in \NN^{\ol{\Sigma}}$ and $\alpha\in \NN_{>0}^{\Sigma}$. We define $\ul{1}:=\sum_{\sigma \in \Sigma} \sigma\in \NN^{\Sigma}$. With this notation, we will write $\alpha\geq \underline{1}$ for the condition $\alpha(\sigma)\geq1$ for all $\sigma\in \Sigma$.
In this paper, we will mainly be interested in the case $\beta=0$, i.e., we will only consider critical Hecke characters of infinity type $-\alpha\in \ZZ^{\Sigma}$ and $\alpha\geq \underline{1}$. In this case, we have
\[
	\{\sigma: \mu(\sigma)\neq 0\}=\Sigma.
\]

\section{$p$-adic Hecke characters}\label{sec:p-adic-Hecke}
Fix once for all a prime number $p$ unramified in $L$ and  an embedding \begin{equation*}
    \iota_p:\ol{\QQ}\hookrightarrow \Cp.
\end{equation*}
Let $\frf\subset\sO_L$ be co-prime to $p$. A \emph{$p$-adic Hecke character} of conductor $p^\infty \frf$ is a continuous group homomorphism \begin{equation*}
    G(p^\infty \frf)=\Gal(L(p^\infty \frf)/L)\to \OCp^\times.
\end{equation*}
 To every algebraic Hecke character $\chi:\cI(p\frf)\to \C^{\times}$ of conductor dividing $p^\infty \frf$ and infinity type $\mu\in \ZZ^{J_L}$, we can associate a $p$-adic Hecke character
\[
	\chi^{\padic}\colon G(p^\infty \frf)\to \OCp^\times,
\]
as follows: Using the embedding $\iota_p$, we consider $\chi$ as a character $\chi \colon \cI(p\frf)\to \Cp^\times$
with
\begin{equation*}\label{eq:p-adic-Hecke-char-2}
	\chi((\lambda))=\lambda^\mu,\quad \text{ for all }\lambda \in \cP_{p\frf},
\end{equation*}
which implies  that $\chi(\cP_{p^n\frf})\subseteq 1+p^n\OCp$.
Together with the finiteness of $\cI(p\frf)/\cP_{p\frf}$, we deduce that $\chi(\cI(p\frf))\subseteq \OCp^\times$. Passing to the limit, one obtains a continuous homomorphism
\[
	\chi\colon \varprojlim_n \cI(p\frf)/\cP_{p^n\frf}\to \OCp^\times.
\]
Finally, by class field theory, the Artin map induces a canonical isomorphism
\[
	\varprojlim_n \cI(p\frf)/\cP_{p^n\frf}\xrightarrow{\sim} G(p^\infty \frf).
\]
\begin{definition}
	Let $\frf$ be an ideal of $\sO_L$ and $\chi$ an algebraic Hecke character of conductor dividing $p^\infty \frf$. The continuous homomorphism 
	\[
	\chi=\chi^{\padic}\colon G(p^\infty \frf)\cong \varprojlim_n \cI(p\frf)/\cP_{p^n\frf}\to \OCp^\times
	\]
	is called the \emph{$p$-adic Hecke character associated to $\chi$}, which by abuse of notation is simply written as $\chi$, if there is no risk of confusion.
\end{definition}

Using the results of section \ref{sec:Hecke-character-variety} of part I, we can view $\chi=\chi^{\padic}$ as a $\Cp$-point of $ \widehat{G(p^\infty\frf)}$.

\begin{lemma}\label{lem:Sigma-analytic-Hecke-char}
	Let $\chi$ be a critical Hecke character of conductor dividing $p^\infty \frf$, CM type $\Sigma$ and of infinity type $-\alpha\in \ZZ^{\Sigma}$ with $\alpha\geq \underline{1}$. Then the $p$-adic Hecke character  associated to  $\chi$ is locally $\Sigma$-analytic in the sense of Definition \ref{def:Sigma-analytic-Galois}.
\end{lemma}
\begin{proof}
	After translation by an element of $\Gal(L(p^\infty \frf)/L)$, it suffices to check that $\chi$ is locally $\Sigma$-analytic in a neighbourhood of $1\in \Gal(L(p^\infty \frf)/L)$. Since $\chi$ is of infinity type $-\alpha$, the composition
	\[
			\pr\circ (\chi|_{\pi^{-1}(\{1\})}) \colon (\sO_L\otimes \Zp)^\times \twoheadrightarrow \pi^{-1}(\{1\})\subseteq \Gal(L(p^\infty \frf)/L)\xrightarrow{\chi} \Cp
		\]
	is given by $z\mapsto z^{-\alpha}$ with $\alpha\in \NN^\Sigma$, which is a $\Sigma$-analytic map.
\end{proof}

\section{The main result}
We keep the previous setup, i.e., $L$ is a totally imaginary field containing a CM field, $p$ is a prime which is unramified in $L$, and $\Sigma$ is a CM type of $L$ induced from a CM field $K\subseteq L$.
\begin{theorem}\label{thmA}
Let  $\frf\subseteq \sO_L$ be a non-trivial ideal which is co-prime to $p$. There is a rigid analytic function $L_p\in \cO(\widehat{G(p^\infty \frf)}_{W(\Sigma)})$ on the $\Sigma$-analytic Hecke character variety with the following interpolation property: For every critical algebraic Hecke character $\chi$ of conductor dividing $p^\infty \frf$, CM type $\Sigma$ and infinity type $-\alpha\in \ZZ^\Sigma$ with $\alpha\geq \underline{1}$, one has
\[
	\frac{L_p(\chi)}{\Omega_p^{\alpha}} =(\alpha-\underline{1})! \Local(\chi,\Sigma) \prod_{\frp\mid p} \left( 1-\frac{\chi(\frp^{-1})}{N\frp} \right)\frac{L_\frf(\chi,0)}{\Omega^{\alpha}},
\] 
where $\Omega$ (resp. $\Omega_p$) are ($p$-adic) periods of an abelian variety $\cA$ with $\cA(\CC)=\CC^\Sigma/\sO_L$, see section \ref{sec:Periods}, and $\Local(\chi,\Sigma)$ is an explicit local factor, see Definition \ref{def:Local-factor}.
\end{theorem}

As a corollary, we can formulate the following consequence which is closer to the 'classical' construction of $p$-adic $L$-functions in the sense of Kubota and Leopoldt.

Let $\frf$ be an ideal of $\sO_L$ co-prime to $p$ and $\chi$ a critical Hecke character of conductor dividing $\frf$ and of infinity type $-\alpha$. We write $\omega_{\chi}$ for the finite $p$-adic Hecke character
\[
	\omega_\chi \colon \Gal(L(p^\infty \frf)/L)\to \sO_{\Cp}^\times\to \ol{\FF}_p^\times \xrightarrow{\omega} \sO_{\Cp}^\times,
\]
where the second map is the reduction map to the residue field, and the lat map is the Teichmüller character. As $\omega_\chi$ is a finite $p$-adic Hecke character, we can use our fixed embeddings $\ol{\QQ}\subseteq \CC$ and $\ol{\QQ}\subseteq \Cp$ to view $\omega_\chi$ as a classical Hecke character of conductor dividing $p^\infty \frf$.

\begin{corollary}
Let $\frf$, $\chi$ and $\omega_\chi$ as above. For any finite Hecke character $\chi_0$ of conductor dividing $p^\infty \frf$ and any positive integer $k$, there exists a unique locally analytic function
\[
	L_p(\cdot, \chi_0 \omega_\chi^k)\colon \Zp \to \Cp, \quad
\]
which is  characterized by the following interpolation formula:
\begin{equation}\label{eq:interpolation-formula}
	\frac{1}{\Omega_p^{k\alpha}}L_p(k, \chi_0\omega_\chi^k) =\frac{(k\alpha-\ul{1})!}{\Omega^{k\alpha}} \prod_{\frp\mid p} \left( 1-\frac{(\chi_0\chi^k)(\frp^{-1})}{N\frp} \right)L_\frf(\chi_0 \chi^k,0).
\end{equation}
\end{corollary}
\begin{proof}
 For $\alpha\in \sO_{\Cp}^\times$, we define $\langle \alpha\rangle:=\frac{\alpha}{\omega(\alpha)}\in 1+\mathfrak{m}_E$. 
We define $L_p(\cdot,\chi_0\cdot \omega_\chi^k )$ as follows
\[
	L_p(s,\chi_0 \omega_\chi^k ):=\frac{1}{ \Local(\chi_0\omega_\chi^k,\Sigma)}L_p (\chi_0 \omega_\chi^k \langle\chi\rangle^s).
\]
It is easily checked that this defines a locally analytic function on $\Zp$. The equality $\langle\chi\rangle \omega_\chi=\chi$ of $p$-adic Hecke characters together with Theorem \ref{thmA} shows, that the function $L_p(\cdot,\chi_0 \omega_\chi^k )$  has the interpolation property
\[
	\frac{1}{\Omega_p^{k\alpha}}L_p(k,\chi_0\omega_\chi^k)=\frac{(k\alpha-\ul{1})!}{\Omega^{k\alpha}}\frac{\Local(\chi_0\chi^k,\Sigma)}{\Local(\chi_0\omega_\chi^k,\Sigma)} \prod_{\frp\mid p} \left( 1-\frac{(\chi_0\chi^k)(\frp^{-1})}{N\frp} \right)L_\frf(\chi_0\chi^k,0).
\]
Since $\chi$ and $\omega_\chi$ are unramified at $p$, we have $(\chi_0\omega_\chi^k)_\fin=(\chi_0)_\fin=(\chi_0\chi^k)_\fin$ , where we use the notation $(\cdot)_{\fin}$ to denote the finite part of a $p$-adic Hecke character, see \eqref{eq:chi-fin}. It follows from the definition of the local factor that $\Local(\chi_0\omega_\chi^k,\Sigma)=\Local(\chi_0\chi^k,\Sigma)$, see \eqref{def:Local-factor}. This shows the interpolation formula \eqref{eq:interpolation-formula}. Note that the interpolation formula for all $k$ and all $\chi_0$ determines the locally analytic function $L_p(\cdot,\chi_0\omega_\chi^k)$ uniquely. Write $q$ for the order of $\omega_\chi$.
If we fix a positive integer $k$ and a finite character $\chi_0$ then we have $\chi_0\omega_\chi^k=\chi_0\omega_\chi^m$ for every $m\equiv k\mod q$. In particular, the interpolation formula \eqref{eq:interpolation-formula} determines $L_p(\cdot,\chi_0\omega_\chi^k)$ on the dense set subset $\{m\in \ZZ_{>0}: m\equiv k\mod q\}$ of $\Zp$. This proves uniqueness.
\end{proof}

\section{Periods, complex and $p$-adic}\label{sec:Periods}
In this section we define the classical and $p$-adic periods, which occur in the main theorem \ref{thmA}.  
We put ourselves in the situation of Notation \ref{not:CMfields}, i.e., $L$ is a totally imaginary field of degree $2d$ with maximal CM subfield $K$ and that $\Sigma$ is a CM type of $L$ induced from  $K$. The prime $p$ is unramified in $L$ and we have fixed embeddings $\iota_\infty: \ol{\QQ} \hookrightarrow \CC$ and $\iota_p\colon \ol{\QQ} \hookrightarrow \CC_p$.

We start by defining the relevant abelian variety with CM by $\sO_L$. The  CM type $\Sigma$ allows us to view $\sO_L$ as a lattice of full rank in $\CC^\Sigma$ using the embedding:
\[
	\sO_L\hookrightarrow \CC^\Sigma, \quad \lambda\mapsto (\sigma(\lambda))_{\sigma \in \Sigma}.
\]
By the theory of complex multiplication, the complex torus $\CC^\Sigma/\sO_L$ consists of the $\C$-valued points of an abelian variety $\cA$ of dimension $d$ with CM by $\sO_L$ defined over $\ol{\QQ}\subset \C$ with good reduction everywhere. Our fixed embedding $\iota_p$ defines a prime ideal $\frP$ of $\ol{\ZZ}$ and we can consider $\cA$ as defined over $R:=\overline{\ZZ}_{\frP}$ the localization of $\ol{\ZZ}$ at $\frP$. Via the fixed embeddings $\iota_\infty$, $\iota_p$ we can consider $R$ as a subring of $\C_p$ and $\C$.

\begin{notation}\label{not-abelian-variety}
We let $\cA$ be the abelian variety with CM by $\sO_L$ defined over $\cR:=\Spec(R)$, with $R:=\overline{\ZZ}_{\frP}$, such that $\cA(\C)\isom \CC^\Sigma/\sO_L$. 
\end{notation}

To define the complex periods we use the period pairing. 
\[
	\langle\cdot, \cdot \rangle_\infty\colon H^1_{\dR}(\cA)\times H_1(\cA(\CC),\ZZ)\to \CC,
\]
where $H^1_{\dR}(\cA)$ is the relative de Rham cohomology.
The canonical isomorphism $\xi_\infty:\sO_L\isom H_1(\cA(\CC),\ZZ)$ gives  a distinguished element in $\xi_p(1)$ in $H_1(\cA(\CC),\ZZ)$. To  define elements in $H^1_{\dR}(\cA/\cR)$ we consider the action of $\sO_L$ on $\Lie(\cA)$. By our choice, $p$ does not divide $d_L$, so that  $R=\ol{\ZZ}_\frP$ contains $\sO_{L^{\Gal}}[1/d_L]$, where $\sO_{L^{\Gal}}$ is the ring of integers of the Galois closure of $L$. Thus  $\Lie(\cA)$ decomposes into eigenspaces corresponding to the different embeddings in $\Sigma$. We have
\begin{equation*}\label{eq:decomp-Lie}
	\Lie(\cA)=\bigoplus_{\sigma\in \Sigma} \Lie(\cA)(\sigma),
\end{equation*}
where $\Lie(\cA)(\sigma)$ is a rank one free $R$-module on which $\gamma\in \sO_L^\times$ acts by multiplication with $\sigma(\gamma)$. By passing to the dual, this decomposition induces a decomposition of the co-Lie algebra $\omega_{\cA}:=\Hom_{R}(\Lie(\cA),R)$ as $R[\sO_L^\times]$-module
\begin{equation*}\label{eq:decomp-coLie}
	\omega_{\cA}=\bigoplus_{\sigma\in \Sigma} \omega_{\cA}(-\sigma).
\end{equation*}
Note that $\gamma\in \sO_L^\times$ acts by $\sigma(\gamma)^{-1}$ on the contragredient representation $\omega_\cA(-\sigma)$ of $\Lie(\cA)(\sigma)$. 
For later use, we introduce the following notation concerning tensor symmetric powers.
\begin{notation}\label{not:Sym-alpha-def}
For $\alpha\in \NN^{\Sigma}$ we write
\begin{equation*}
    \TSym^{\alpha}(\omega_{\cA}):=\bigotimes_{\sigma\in \Sigma}\TSym^{\alpha(\sigma)}(\omega_{\cA}(-\sigma))
\end{equation*}
where $\TSym^{\alpha(\sigma)}(\omega_{\cA}(-\sigma))$ is the $\alpha(\sigma)$-th tensor symmetric power. 
\end{notation}
For each $\sigma\in \Sigma$, let us fix an $R$-basis $\omega(\cA)(\sigma)\in  \omega_{\cA}(-\sigma)$ of the free $R$-module $\omega_{\cA}(-\sigma)$. We write 
\begin{equation*}
    \omega(\cA):=(\omega(\cA)(\sigma))_\sigma
\end{equation*} 
for the basis of $\omega_{\cA}$. The Hodge filtration 
\[
	\omega_{\cA}\subseteq H^1_{\dR}(\cA)
\]
allows us to view each $\omega(\cA)(\sigma)$ as an element of $H^1_{\dR}(\cA)$.
\begin{definition}\label{def:periods}
With the  basis $\omega(\cA)$ of $\omega_{\cA}$ we define complex periods
\[
	\Omega(\sigma):=\langle \omega(\cA)(\sigma),\xi_\infty(1) \rangle_\infty\in \CC
\]
and we write $\Omega=(\Omega(\sigma))_{\sigma \in \Sigma}\in \CC^\Sigma$. For $\alpha\in \NN^\Sigma$ we write 
\begin{equation*}
    \Omega^\alpha:=\prod_{\sigma\in \Sigma} \Omega(\sigma)^{\alpha(\sigma)}. 
\end{equation*}
\end{definition}

We now define the $p$-adic periods. As $\cA(\C)\isom \C^{\Sigma}/\sO_L$, we have a canonical isomorphism
\begin{equation*}\label{eq:basis-xi-n}
	p^{-n}\cO_L/\cO_L= \cA[p^n](\CC)\cong \cA[p^n](\overline{\QQ}).
\end{equation*}
By passing to the limit, we obtain a $\sO_L\otimes \Zp$-basis
$	\xi_p\in T_p\cA$ of $T_p\cA$. Let us denote by $\zeta_{p^n}\in \overline{\QQ}$ the unique primitive $p^n$-th root of unity mapping to $\exp(2\pi i/p^n)$ under the fixed embedding $\overline{\QQ}\hookrightarrow \CC$. This compatible system of $p$-power roots of unities gives an isomorphism 
\begin{equation*}\label{eq:Tate-generator}
	\Zp\xrightarrow{\sim} \Zp(1), \quad 1\mapsto (\zeta_{p^n})_n.
\end{equation*}
The bases $\xi_p$ and $(\zeta_{p^n})_n$ induce a canonical $\sO_L\otimes \Zp$-basis $\xi_p^\vee \in T_p\cA^\vee$ of $T_p\cA^\vee$ under the isomorphism $T_p \cA^\vee\cong \Hom(T_p\cA,\Zp(1))$. Let us write \begin{equation}\label{eq:Lie-basis}
    \LieA=(\LieA(\sigma))_{\sigma\in \Sigma}
\end{equation} 
for the $R$-basis of $\Lie(\cA)=\bigoplus_{\sigma\in \Sigma} \Lie(\cA)(\sigma)$ dual to the fixed basis $\omega(\cA)=(\omega(\cA)(\sigma))_\sigma$.
The $p$-adic period pairing 
\begin{equation*}
    \langle-,-\rangle_p\colon \Lie(\cA)\times T_p\cA^\vee \to \OCp
\end{equation*}
allows to define the $p$-adic periods. 

\begin{definition}\label{def:p-adic-periods}
	The $p$-adic period $\Omega_p:=(\Omega_{p}(\sigma))_{\sigma\in \Sigma}\in \CC_p^\Sigma$ with respect to the basis $\xi_p^\vee$ of $T_p\cA^\vee$ and $\LieA$ of $\Lie(\cA)$ is defined as
	\[
		\Omega_{p}(\sigma):=\langle \LieA(\sigma),\xi_p^\vee \rangle_p.
	\]
	For $\alpha\in \NN^\Sigma$, we will write $\Omega_p^{\alpha}:=\prod_{\sigma\in \Sigma} \Omega_{p}(\sigma)^{\alpha(\sigma)}$.
\end{definition}

\section{Serre twists and decompositions}
Let $\cA/\cR$ be the abelian variety with CM by $\sO_L$ and CM type $\Sigma$ defined in Notation \ref{not-abelian-variety}
and $\Gamma\subset \sO_L^{\times}$ a subgroup of finite index. Note that $\Gamma$ acts on $\cA/\cR$. 
For any fractional ideal $\fra$ of $L$, we define the Serre twist $\cA(\fra)$ of $\cA$ by 
\begin{equation*}
    \cA(\fra):=\fra\otimes \cA.
\end{equation*}
\begin{notation}
In this section we let $\cB/\cR$ denote a Serre twist of $\cA$ by some fractional ideal $\fra$, which is \emph{coprime} to  $p$.
\end{notation}
The Serre twist satisfies $\cA(\fra)(\frb)\isom \cA(\fra\frb)$ and is
is functorial in $\fra$. In particular, for an integral ideal $\frc\subset \sO_L$, which is coprime to $p$,  we denote by 
\begin{equation*}
    [\frc]:\cB\to \cB(\frc^{-1})
\end{equation*}
the resulting \'etale (as $\cR=\Spec \ol{\Z}_\frP$) isogeny of degree $N\frc$ and denote by $\cB[\frc]:=\ker[\frc]$ the $\frc$-torsion points.
Let $\omega_{\cB}$ be the dual of $\Lie(\cB)$ and 
\begin{equation*}
    \omega_{\cB}^{d}:=\Lambda^{d}\omega_{\cB}.
\end{equation*}
Notice that $\omega_{\cB}\isom e^{*}\Omega^{1}_{\cB}$, where $e$ is the unit section of $\cB$.
By our assumption on $\cB$, there exists a natural $\sO_L$-linear \'etale quasi-isogeny connecting $\cA$ and $\cB$. More explicitly, given a fractional ideal $\fra$, we can write $\fra$ uniquely as $\fra=\fra_0\fra_1^{-1}$ with $\fra_0,\fra_1\subseteq \sO_L$ co-prime, and get a natural quasi-isogeny
\[
	\cB=\cA(\fra)\leftarrow \cA(\fra_0)\rightarrow \cA. 
\]
In particular, the
decomposition in \eqref{eq:decomp-coLie} induces a decomposition
\begin{equation*}
    \omega_{\cB}\isom \bigoplus_{\sigma\in \Sigma}\omega_{\cB}(-\sigma),
\end{equation*}
and the basis $\omega(\cA)(\sigma)$ induces by pullback a basis $\omega(\cB)(\sigma)$ of $\omega_{\cB}(-\sigma)$.
Recall that $\gamma\in \Gamma$ acts by multiplication with $\sigma(\gamma)^{-1}$ on $\omega_{\cB}(-\sigma)$, and that we have introduced for $\alpha\in \NN^{\Sigma}$ the notation  
\begin{equation*}\label{eq:Sym-alpha-def}
    \TSym^{\alpha}(\omega_{\cB}):=\bigotimes_{\sigma\in \Sigma}\TSym^{\alpha(\sigma)}(\omega_{\cB}(-\sigma))
\end{equation*}
where $\TSym^{\alpha(\sigma)}(\omega_{\cB}(-\sigma))$ is the $\alpha(\sigma)$-th tensor symmetric power. With this convention, one gets 
\begin{equation*}\label{eq:Sym-a-decomposition}
    \TSym^{a}(\omega_{\cB})\isom \bigoplus_{|\alpha|=a}\TSym^{\alpha}(\omega_{\cB}).
\end{equation*}
We use also the notation $\Sym^{\alpha}(\Lie(\cB))$ which is defined in the same way. 
For later use, we define a basis of $\TSym^{\alpha}(\omega_{\cB}(-\sigma))$
by
\begin{equation}\label{eq:m-alpha}
    \omega(\cB)^{[\alpha]}:=\bigotimes_{\sigma\in \Sigma}\omega(\cB)(\sigma)^{\otimes\alpha(\sigma)}
\end{equation}
where we used that 
$\omega(\cB)=\prod_{\sigma\in \Sigma}\omega(\cB)(\sigma)$. 
The following technical proposition will play a crucial role in the evaluation of the Eisenstein-Kronecker classes.
\begin{proposition}\label{prop:cap-product}Consider $\cA/\cR$.
Let $M$ be a  $R[\Gamma]$-module, such that $\Gamma$ acts trivially on $M\otimes\omega_{\cA}^{d}$. Then for any choice of ordering of $\Sigma$ there is a homomorphism
\begin{equation*}
    H^{d-1}(\Gamma, M\otimes\omega_{\cA}^{d})\to M\otimes \TSym^{\ul{1}}(\omega_{\cA}),
\end{equation*}
which is independent of this choice.
\end{proposition}
\begin{proof}The proof is similar to  the one of \cite{Kings-Sprang} Proposition 2.27. Note that $ H^{d-1}(\Gamma, M\otimes\omega_{\cA}^{d})= H^{d-1}(\Gamma, R)\otimes M\otimes\omega_{\cA}^{d}$.
Let $\Gamma'\subset \Gamma$ be a free subgroup of finite index. We identify
\begin{equation*}
    H_{d-1}(\Gamma',R)\isom H_{d-1}(\R_{>0}^{\Sigma}/\Gamma',R)
\end{equation*}
and the ordering of $\Sigma$ together with the standard orientation of $\R$ gives an orientation of $\R_{>0}^{\Sigma}/\Gamma'$ and hence a fundamental class $\xi'\in H_{d-1}(\Gamma',R)$. The same ordering gives an isomorphism
$\omega_{\cA}^{d}\isom\TSym^{\ul{1}}(\omega_{\cA})$ of $\Gamma$-modules.
As in the proof of \cite{Kings-Sprang} Proposition 2.27. one shows that 
\begin{equation*}
    H_{d-1}(\Gamma,R)\xrightarrow{\res} H_{d-1}(\Gamma',R)
\end{equation*}
is surjective, so that there is a $\xi\in H_{d-1}(\Gamma,R)$ with $\res(\xi)=\xi'$. Using this $\xi$ one has a commutative diagram
\begin{equation*}
    \xymatrix{
H^{d-1}(\Gamma,R)\otimes M\otimes \omega^{d}_{\cA}\ar[d]_\res\ar[r]^{\xi\cap- }&H_0(\Gamma,M\otimes \TSym^{\ul{1}}(\omega_{\cA}))\ar[d]^{[\Gamma:\Gamma']}\\
 H^{d-1}(\Gamma',R)\otimes M\otimes \omega^{d}_{\cA}\ar[r]^{\res(\xi)\cap-}&H_0(\Gamma',M\otimes \TSym^{\ul{1}}(\omega_{\cA}))}
\end{equation*}
because for any 
$\eta\in H^{d-1}(\Gamma,R)\otimes M\otimes \omega^{d}_{\cA}$, the projection formula gives
\begin{equation*}
    \cor(\res(\xi)\cap\res(\eta))=\cor\circ\res(\xi)\cap \eta=[\Gamma:\Gamma']\xi\cap \eta.
\end{equation*}
As $R$ is torsion free, multiplication by $[\Gamma:\Gamma']$ is injective, which shows that the map $\xi\cap-$ is independent of the choice of $\xi$ and the ordering of $\Sigma$.
\end{proof}

\section{Eisenstein-Kronecker classes and the class $\vartheta_{\frf,\frc}(\frb)$}\label{section:EK-classes}
Let us fix $p$ as before and ideals $\frf,\frc,\frb$ of $\sO_L$ not divisible by $p$ and pairwise coprime.  We  write $\cB:=\cA(\frf\frb^{-1})$ for the Serre twist of $\cA$ by $\frf\frb^{-1}$. We denote by $\cB[\frc]$ the kernel of the isogeny $[\frc]:\cB\to \cB(\frc^{-1})$ and we let 
\begin{equation*}
    \Gamma:=\sO_\frf^{\times}
\end{equation*}
be the group of units, that are congruent to $1$ modulo $\frf$.
The uniformization $\cA(\C)\isom \C^{\Sigma}/\sO_L$ induces  uniformizations $\cA(\frf)(\C)\isom  \C^{\Sigma}/\frf$ and $\cB(\C)\isom \C^{\Sigma}/\frf\frb^{-1}$. With this we get (using $\iota_\infty:\ol{\QQ}\to \C$)
\begin{align}\label{eq:not:x-Omega}
    \sO_L/\frf\isom\cA(\frf)[\frf](\cR)&&\mbox{ and }&&\frc^{-1}\frf\frb^{-1}/\frf\frb^{-1}\isom \cB[\frc](\cR)
\end{align}
which we consider as an identification. Thus $1\in \sO_L/\frf$ is a $[\frf]$-torsion point in $\cA[\frf](\C)$.

We now review the Eisenstein-Kronecker classes constructed in \cite{Kings-Sprang}.
\begin{definition}\label{def:Da}For any fractional ideal $\Lambda$ of $L$ and ideal $\frN\subset \sO_L$ let 
\begin{equation*}
    C(\frN^{-1}\Lambda/\Lambda,R)
\end{equation*}
be the $R$-valued functions $f$ on  $\frN^{-1}\Lambda/\Lambda$ with the $\Gamma$-action $\gamma f(\lambda)=f(\gamma^{-1}\lambda)$.
We define  $f_{[\frc]}\in H^{0}(\cB[\frc],\cO_{\cB[\frc]})\isom C(\frc^{-1}\frf\frb^{-1}/\frf\frb^{-1},R)$ by 
\begin{equation*}\label{eq:f-def}
    f_{[\frc]}(t):=\begin{cases}
-1&t\neq 0\\
N\frc-1&t=0.
\end{cases}
\end{equation*}
where  
$N\frc$ is the norm of $\frc$.
\end{definition}
One has
\begin{equation*}
    f_{[\frc]}\in \ker(H^{0}(\cB[\frc],\cO_{\cB[\frc]})\to H^{0}(\cR,\cO_\cR))^{\Gamma},
\end{equation*}
where $H^{0}(\cB[\frc],\cO_{\cB[\frc]})\to H^{0}(\cR,\cO_\cR)$ is the trace map.
The subscheme  $\cB[\frc]\subset \cB$ is  non-empty, closed and finite \'etale over $\cR$ and stable under the action of $\Gamma$. In \cite[Corollary 2.33]{Kings-Sprang} it is shown that
\begin{equation}\label{eq:H0-inclusion}
    H^{0}(\cB[\frc],\cO_{\cB[\frc]})^{\Gamma}\subset H^{d}_{\cB[\frc]}(\cB,\Gamma;\Omega^{d}_{\cB}),
\end{equation}
where the right hand side denotes $\Gamma$-equivariant sheaf cohomology with support in $\cB[\frc]$. (We refer to \cite{Kings-Sprang} Appendix A for more details on $\Gamma$-equivariant cohomology). Let $\cU_\frc:=\cB\setminus \cB[\frc]$, then there is a localization sequence
\begin{equation*}
    \ldots\to H^{d-1}(\cU_{\frc},\Gamma;\Omega^{d}_{\cB})\xrightarrow{\res} H^{d}_{\cB[\frc]}(\cB,\Gamma;\Omega^{d}_{\cB})\to H^{d}(\cB,\Gamma;\Omega^{d}_{\cB})\to \ldots
\end{equation*}
In this set-up, one can construct the Eisenstein-Kronecker class following \cite{Kings-Sprang}. 
\begin{theorem}[\cite{Kings-Sprang} Theorem 2.20]There is a unique class 
\begin{equation*}
    {_{\frc}}\EK_{\Gamma,\cB}\in H^{d-1}(\cU_\frc,\Gamma;\Omega^{d}_{\cB})
\end{equation*}
such that $\res\circ {_{\frc}}\EK_{\Gamma,\cB}=f_{[\frc]}$ and which lifts to the completion of the Poincar\'e bundle. Here we consider 
\begin{equation*}
    f_{[\frc]}\in H^{0}(\cB[\frc],\cO_{\cB[\frc]})^{\Gamma}\subset H^{d}_{\cB[\frc]}(\cB,\Gamma;\Omega^{d}_{\cB}),
\end{equation*} 
using \eqref{eq:H0-inclusion}.
The class $ {_{\frc}}\EK_{\Gamma,\cB}$ is called the \emph{equivariant Eisenstein-Kronecker class}.
\end{theorem}
\begin{remark}
This class is denoted by $\EK_{\Gamma,\cB}(f_{[\frc]})$ in \cite{Kings-Sprang}. As we use only $f_{[\frc]}$ in this paper, we have slightly changed the notation.
\end{remark}
\begin{proof}
Theorem 2.20 in \cite{Kings-Sprang} with $f=f_{[\frc]}$ gives a class 
in $H^{d-1}(\cU_\frc,\Gamma;\wP\otimes\Omega^{d}_{\cB})$, which is uniquely determined by its residue being $f_{[\frc]}$ with values in the completion of the Poincar\'e bundle $\wP$. The class here is the image  under the augmentation map $\wP\to \sO_{\cB}$.
\end{proof}

As in \cite{Kings-Sprang} we use differentiation to produce 
further classes out of $ {_{\frc}}\EK_{\Gamma,\cB}$. Let $\alpha\in \NN^\Sigma$ and write $a:=|\alpha|:=\sum_{\sigma}\alpha(\sigma)$. Using the differential $d:\sO_{\cB}\to \Omega^{1}_{\cB}$ iteratively $a$-times and projecting to the $\TSym^{\alpha}(\Omega^{1}_\cB)$ component gives  a map of sheaves of $R$-modules
\begin{equation*}
    d^{\alpha}:\sO_{\cB}\to \TSym^{\alpha}(\Omega^{1}_{\cB})
\end{equation*}
and hence a class
\begin{equation*}
    d^{\alpha}{_{\frc}}\EK_{\Gamma,\cB}\in H^{d-1}(\cU_\frc,\Gamma;\TSym^{\alpha}(\Omega^{1}_{\cB})\otimes\Omega^{d}_{\cB}).
\end{equation*}
We now define a map $\tau_{n}:\cA[p^{n}]\to \cU_\frc$
and we will pull-back the 
class $d^{\alpha}{_{\frc}}\EK_{\Gamma,\cB}$ by this map. 
The isogeny 
$[\frb]:\cA(\frf)\to \cB:=\cA(\frf\frb^{-1})$ has degree $N\frb$, which is coprime to $\frf$ and hence $[\frb](1)\in \cB(\cR)$ is also a primitive $\frf$-torsion point,  where $1\in \cA(\frf)[\frf](\cR)$ (cf. \eqref{eq:not:x-Omega}). Explicitly, the point $[\frb](1)$ corresponds to $1\in \frb^{-1}/\frf\frb^{-1} \cong \cB[\frf](\cR)$.
The isogeny $[\frf]:\cA(\frf)\to \cA$ induces an isomorphism $\cA(\frf)[p^{n}]\isom \cA[p^{n}]$ as $\frf$ is coprime to $p$. Let $\tau_{n}$ be the composition 
\begin{equation*}
    \tau_{n}:\cA[p^{n}]\isom \cA(\frf)[p^{n}]\xrightarrow{[\frb]}\cB[p^{n}]\xrightarrow{T_{[\frb](1)}}\cB\setminus\cB[\frc]=\cU_\frc,
\end{equation*}
where $T_{[\frb](1)}$ is the translation by $[\frb](1)$. Then $\tau_{n}$ is $\Gamma$-equivariant. The pull-back of
$d^{\alpha}{_{\frc}}\EK_{\Gamma,\cB}$ by 
$\tau_{n}$ is a cohomology class in 
\begin{align*}
    H^{d-1}(\cA[p^{n}], \Gamma;\tau_{n}^{*}(\TSym^{\alpha}(\Omega^{1}_{\cB})\otimes\Omega^{d}_{\cB}))&\isom 
H^{d-1}(\Gamma,H^{0}(\cA[p^{n}], \TSym^{\alpha}(\Omega_\cA)\otimes\Omega^{d}_{\cA}))
\\&\isom H^{d-1}(\Gamma,\cO(\cA[p^{n}])\otimes\TSym^{\alpha}(\omega_\cA)\otimes \omega^{d}_{\cA}),
\end{align*}
where the first isomorphism comes from the fact that $\cA[p^{n}]$ is affine. 
Using further the identification 
$R\isom \TSym^{\ul{1}}(\omega_{\cA})\otimes \Sym^{\ul{1}}(\Lie(\cA))$, one can rewrite this as
\begin{equation*}
H^{d-1}(\Gamma,\cO(\cA[p^{n}])\otimes\TSym^{\alpha+\ul{1}}(\omega_\cA)\otimes\Sym^{\ul{1}}(\Lie(\cA))\otimes \omega^{d}_{\cA}).
\end{equation*}
From the construction it is clear that
\begin{equation*}
    \tau_{n+1}^{*}(d^{\alpha}{_{\frc}}\EK_{\Gamma,\cB})\mid_{\cA[p^{n}]}=\tau_{n}^{*}(d^{\alpha}{_{\frc}}\EK_{\Gamma,\cB})
\end{equation*}
and we can define:
\begin{definition}\label{def:vartheta}Let $n$ be a positive integer and $\alpha\ge \ul{1}$. Then we write
\begin{equation*}
    \vartheta_{\frf,\frc,n}^{\alpha}(\frb):=\tau_{n}^{*}(d^{\alpha-\ul{1}}{_{\frc}}\EK_{\Gamma,\cB})\in H^{d-1}(\Gamma,\cO(\cA[p^{n}])\otimes\TSym^{\alpha}(\omega_\cA)\otimes\Sym^{\ul{1}}(\Lie(\cA))\otimes \omega^{d}_{\cA}))
\end{equation*}
(notice the shift in $\alpha$). If $\alpha=\ul{1}$ we write
\begin{equation*}
    \vartheta_{\frf,\frc,n}(\frb)=\vartheta^{\ul{1}}_{\frf,\frc,n}(\frb).
\end{equation*}
\end{definition}

\section{$\vartheta_{\frf,\frc,n}(\frb)$ and $L$-values}\label{sec:vartheta}
We keep the notations from section \ref{section:EK-classes}.
We want to evaluate $ \vartheta^{\alpha}_{\frf,\frc,n}(\frb)$ on certain functions on $\cB[p^{n}](R)$ with values in $R$. Before we do this,
we consider the evaluation process in the more general setting for an abelian variety $\cB$ such that $\cB(\C)\isom \C^{\Sigma}/\Lambda$ and where $\Lambda$ is a fractional ideal of $L$.
Then one has a $\Gamma$-equivariant pairing 
\begin{align*}
    C(p^{-n}\Lambda/\Lambda,R)\times \cO(\cB[p^n])\to R&&(f,g)\mapsto \sum_{\lambda\in p^{-n}\Lambda/\Lambda}f(\lambda)\lambda^{*}g
\end{align*}
where 
$\lambda^*$ is the pullback along the torsion section $\lambda\in p^{-n}\Lambda/\Lambda\cong \cB[p^n](\ol{\QQ})$. 
\begin{proposition}\label{prop:evaluation-map}
For each $\Gamma$-invariant $\rho\otimes \Xi\in C(p^{-n}\Lambda/\Lambda,R)\otimes \Sym^{\alpha}(\Lie(\cB))$, there is an evaluation map 
\begin{equation*}
    \eval_{\rho\otimes\Xi,\Gamma}:H^{d-1}(\Gamma,\cO(\cB[p^n])\otimes\TSym^{\alpha}(\omega_\cB)\otimes\Sym^{\ul{1}}(\Lie(\cB))\otimes\omega^{d}_{\cB})\to R.
\end{equation*}
If $\Gamma'\subset \Gamma$ is a subgroup of finite index, one has
\begin{equation*}
    \eval_{\rho\otimes\Xi,\Gamma'}(\res(\eta))=[\Gamma:\Gamma']\eval_{\rho\otimes\Xi,\Gamma}(\eta).
\end{equation*}
\end{proposition}
\begin{proof}
The cup-product with $\rho\otimes\Xi\in(C(p^{-n}\Lambda/\Lambda,R)\otimes \Sym^{\alpha}(\Lie(\cB)))^{\Gamma}$ gives a map
\begin{equation*}
    H^{d-1}(\Gamma,\cO(\cB[p^n])\otimes\TSym^{\alpha}(\omega_\cB)\otimes\Sym^{\ul{1}}(\Lie(\cB))\otimes\omega^{d}_{\cB})\xrightarrow{\rho\otimes\Xi\cup} H^{d-1}(\Gamma,\Sym^{\ul{1}}(\Lie(\cB))\otimes\omega^{d}_{\cB}),
\end{equation*}
which composed with Proposition \ref{prop:cap-product} gives the desired evaluation map. Let $\xi\in H_{d-1}(\Gamma,R)$ be as  in the proof of \ref{prop:cap-product}. Then 
\begin{align*}
    \eval_{\rho\otimes\Xi,\Gamma'}(\res(\eta))&=\cor(\res(\xi)\cap(\res((\rho\otimes\Xi)\cup \res(\eta)))\\
&=(\cor\circ\res(\xi))\cap((\rho\otimes \Xi)\cup \eta) \\
&=[\Gamma:\Gamma']\eval_{\rho\otimes\Xi,\Gamma}(\eta)
\end{align*}
and the formula $ \eval_{\rho\otimes\Xi,\Gamma'}(\res(\eta))=[\Gamma:\Gamma']\eval_{\rho\otimes\Xi,\Gamma}(\eta)$ follows.
\end{proof}
To get an explicit formula for $\eval_{\rho\otimes\Xi,\Gamma}(\vartheta_{\frf,\frc,n}^{\alpha}(\frb))$ we will need the relation between the Eisenstein-Kronecker classes and Eisenstein series established in \cite{Kings-Sprang}. We consider the compactly supported functions $C_c(\QQ\otimes \Lambda/\Lambda,R)$. One has a canonical isomorphisms
\begin{multline*}\label{eq:C-Lambda}
	\varinjlim_\frN C(\frN^{-1}\Lambda/\Lambda,R)\cong C_c(\QQ\otimes\Lambda/\Lambda,R)\\
	\cong \{f:\Q\otimes \Lambda\to R\mid f \mbox{ is $\Lambda$-invariant and $f\colon \Q\otimes\Lambda/\Lambda\to R$ has finite support}\}.
\end{multline*}
We will simply write $f(\lambda)$ for the value of $C_c(\QQ\otimes\Lambda/\Lambda,R)$ seen as a $\Lambda$-invariant function on $\lambda\in \QQ\otimes \Lambda$. For $x\in \QQ\otimes\Lambda$, we will write \begin{equation*}
    \delta_x\in C_c(\QQ\otimes\Lambda/\Lambda,R)
\end{equation*} 
for the characteristic function of $x+\Lambda\subseteq \QQ\otimes \Lambda$.
\begin{definition}For each $f\in C_c(\QQ\otimes\Lambda/\Lambda,R)$ with $f(\gamma\lambda)=\gamma^{\alpha}f(\lambda)$ we define the Eisenstein series for $s\in \C$ with $\Re(s)>-\frac{a}{2}+d$ by
\begin{equation*}
    E^{0,\alpha}(f,s,\Lambda,\Gamma):=\dashsum_{\lambda\in \Gamma\backslash\Q\otimes\Lambda}\frac{f(\lambda)}{\lambda^{\alpha}N(\lambda)^{s}},
\end{equation*}
where  $\lambda$ runs over all non-zero $\Gamma$-cosets of $\Q\otimes\Lambda$. Here $\lambda^{\alpha}=\prod_{\sigma \in \Sigma}\sigma(\lambda)^{\alpha(\sigma)}$ and $N(\lambda)=\prod_{\sigma \in \Hom(L,\overline{\QQ})}\sigma(\lambda)$ is the usual norm. If $t\in \Q\otimes\Lambda$ is fixed by $\Gamma$ and $\gamma^{\alpha}=1$ for all $\gamma\in \Gamma$, i.e., $\Gamma\subset \ker\alpha$, we write 
\begin{equation*}
     E^{0,\alpha}(t,s,\Lambda,\Gamma):= E^{0,\alpha}(\delta_{t},s,\Lambda,\Gamma).
\end{equation*}
\end{definition}
This Eisenstein series satisfies the following formula with respect to the change of $\Gamma$:
\begin{proposition}\label{prop:eis-and-finite-index}
For a subgroup of finite index $\Gamma'\subset \Gamma$, one has 
\begin{equation*}
    E^{0,\alpha}(f,s,\Lambda,\Gamma')=[\Gamma:\Gamma']E^{0,\alpha}(f,s,\Lambda,\Gamma).
\end{equation*}
\end{proposition}
\begin{proof}
As $\frac{f(\lambda)}{\lambda^{\alpha}N(\lambda)^{s}}$ is $\Gamma$-invariant,  one has
\begin{equation*}
    \sum_{\gamma\in \Gamma/\Gamma'}\frac{f(\gamma\lambda)}{(\gamma\lambda)^{\alpha}N(\gamma\lambda)^{s}}=[\Gamma:\Gamma']\frac{f(\lambda)}{\lambda^{\alpha}N(\lambda)^{s}},
\end{equation*}
which implies the result.
\end{proof}
\begin{corollary}
Let $f\in  C_c(\QQ\otimes\Lambda/\Lambda,R)$ satisfy $f(\gamma^{-1}\lambda)=\gamma^{\alpha}f(\lambda)$ and $\Gamma'\subset \Gamma$ a subgroup of finite index contained in $\ker\alpha$ and fixing $f$. Then
\begin{equation*}
    E^{0,\alpha}(f,s,\Lambda,\Gamma)=\frac{1}{[\Gamma:\Gamma']}\sum_{t\in \Q\otimes\Lambda/\Lambda}f(t)E^{0,\alpha}(t,s,\Lambda,\Gamma').
\end{equation*}
\end{corollary}
\begin{proof}
This follows from writing $f=\sum_{t\in \Q\otimes\Lambda/\Lambda}f(t)\delta_{t}$.
\end{proof}
\begin{remark}The Eisenstein series $E^{0,\alpha}(t',s,\Lambda,\Gamma)$ defined in \cite{Kings-Sprang} 3.24 is only defined if $\Gamma\subset \ker\alpha$ and is equal to $E^{0,\alpha}( t',s,\Lambda,\Gamma)$ in the above notation.
\end{remark}
For the relation of the Eisenstein series with the $L$-series of a Hecke character $\chi$ of conductor dividing $\frf$ and $\infty$-type $-\alpha$, consider the partial $L$-functions
\begin{align*}
    L(\chi,s)=\sum_{[\frb]\in \cI(\frf)/\cP_\frf}L_\frf(\chi,s,[\frb]),&&\mbox{where }&&L_\frf(\chi,s,[\frb]):=\sum_{\stackrel{\fra\in [\frb]}{\fra\subset \sO_L}}\frac{\chi(\fra)}{N\fra^{s}},
\end{align*}
and $[\frb]$ runs over all classes in the ray class group $\cI(\frf)/\cP_\frf$. Then one has 
\begin{equation*}\label{eq:L-and-Eis}
    L_\frf(\chi,s,[\frb])=\chi(\frb)N\frb^{-s} E^{0,\alpha}(1,s,\frf\frb^{-1},\sO_\frf^{\times}),
\end{equation*}
where the right hand side is well-defined, because the condition on $\chi$ in \eqref{eq:p-adic-Hecke-char-2} implies that $\sO_\frf^{\times}\subset \ker\alpha$.

We now want to evaluate $\vartheta^\alpha_{\frf,\frc,n}$, which is the pull-back of an Eisenstein-Kronecker class living on 
$\cB=\cA(\frf\frb^{-1})$. As $\cB(\C)\isom \C^{\Sigma}/\frf\frb^{-1}$ we  let $\Lambda=\frf\frb^{-1}$ and we consider the evaluation map from Proposition \ref{prop:evaluation-map}, where $\frf,\frb,p$ are pairwise coprime as in Section \ref{section:EK-classes}.  For $f_1,f_2\in C_c(\QQ\otimes\Lambda/\Lambda,R)$ the convolution $f_1\ast f_2\in C_c(\QQ\otimes\Lambda/\Lambda,R)$ is defined as
\[
	(f_1\ast f_2)(z)=\sum_{\substack{x,y\in \Q\otimes \Lambda/\Lambda: \\ x+y=z}} f_1(x)f_2(y).
\] 
This sum is well-defined as $f_1,f_2\colon \QQ\otimes \Lambda/\Lambda\to R$ have finite support. We consider  for $x=1\in \frf^{-1}\Lambda$ the function $\delta_1\in C(\frf^{-1}\Lambda/\Lambda,R)\subseteq C_c(\QQ\otimes\Lambda/\Lambda,R)$.
In the following theorem we also write  $\rho\ast\delta_{1}\in C(p^{-n}\frf^{-1}\frc^{-1}\Lambda/\frc^{-1}\Lambda,R)$ for the corresponding function under the canonical isomorphism $C(p^{-n}\frf^{-1}\frc^{-1}\Lambda/\frc^{-1}\Lambda,R)\cong C(p^{-n}\frf^{-1}\Lambda/\Lambda,R)$.

\begin{theorem}\label{thm:evaluation-formula}
Let $\frf$, $\frb$, $\frc$, $p$ be pairwise coprime  and $\Lambda=\frf\frb^{-1}$. Let
$\Xi:=\LieA^\alpha\in  \Sym^{\alpha}(\Lie(\cA))$ be the dual basis of 
the basis $\omega(\cA)^{[\alpha]}$ of $\TSym^{\alpha}(\omega^{1}_{\cA})$  from \eqref{eq:m-alpha} and $\rho:p^{-n}\Lambda/\Lambda\to R$ be a function such that $\rho\otimes\Xi$ is $\sO_\frf^\times$-invariant.
Then one has in $\C$  (via $\iota_\infty$)
\begin{multline*}
     \eval_{\rho\otimes\Xi}(\vartheta_{\frf,\frc,n}^{\alpha}(\frb))\\=(-1)^{\frac{d(d-1)}{2}} \frac{(\alpha-\ul{1})!}{\Omega^\alpha}\cdot
 \left(N\frc E^{0,\alpha}(\rho\ast\delta_1,0,\frf\frb^{-1},\sO_\frf^{\times})-E^{0,\alpha}(\rho\ast\delta_1,0,\frc^{-1}\frf\frb^{-1},\sO_\frf^{\times})\right),
\end{multline*}
where $\rho\ast\delta_{1}\in C(p^{-n}\frf^{-1}\Lambda/\Lambda,R)\subseteq C_c(\QQ\otimes\Lambda/\Lambda,R)$ is the convolution of $\rho$ and $\delta_1$ and where for the second Eisenstein series we use for $\rho\ast\delta_1$ the convention explained above.
\end{theorem}

\begin{proof}By Proposition \ref{prop:evaluation-map} and Proposition \ref{prop:eis-and-finite-index} we may replace $\sO_\frf^\times$ by a finite index subgroup $\Gamma$ which acts trivially on $p^{-n}\frc^{-1}\frb^{-1}\Lambda/\Lambda$. Furthermore, both sides of the equation in the statement of the theorem depend $R$-linearly on $\rho$. Hence it suffices to prove the formula for $\rho=\delta_r$ with $r\in p^{-n}\Lambda/\Lambda$. Write $x\in \cA[p^n\frf](\cR)$ for the torsion point corresponding $1+r\in \Q\otimes\Lambda$. As $\tau_{n}$ maps $p^{-n}\sO_L/\sO_L$ to $1+p^{-n}\Lambda/\Lambda$ the evaluation $\eval_{\rho\otimes\Xi}(\vartheta_{\frf,\frc,n}^{\alpha}(\frb))$ coincides with the image of $d^\alpha {_\frc}\EK_{\Gamma,\cB}$ under the map
\begin{equation}\label{eq:pullback-EK}
	H^{d-1}(\cU_\frc,\Gamma; \TSym^{\alpha}(\Omega^1_\cB)\otimes \Omega^d_\cB)\xrightarrow{x^*} H^{d-1}(\Gamma,\TSym^\alpha(\omega_\cA)\otimes \omega_\cA^d)\cong R,
\end{equation}
where the last isomorphism is induced by $\Xi$ and Proposition \ref{prop:cap-product}. The image $d^\alpha {_\frc}\EK_{\Gamma,\cB}$ under \eqref{eq:pullback-EK} has been computed in \cite[Corollary 3.28]{Kings-Sprang}, where it was denoted by $\EK_\Gamma^{0,\alpha}(f_\frc,x)$. Using Corollary 3.28 of \cite{Kings-Sprang}, one obtains:
\begin{align*}
    \eval_{\rho\otimes\Xi}(\vartheta_{\frf,\frc,n}^{\alpha}(\frb))&=(-1)^{\frac{d(d-1)}{2}}\frac{(\alpha-\ul{1})!}{\Omega^{\alpha}}\sum_{s\in \frc^{-1}\Lambda/\Lambda}f_{[\frc]}(-s)E^{0,\alpha}(1+r+s,0,\frf\frb^{-1},\Gamma).
\end{align*}
As in \cite[Proposition 4.12]{Kings-Sprang} the distribution relation for the Eisenstein series gives 
\begin{equation*}
    \eval_{\rho\otimes\Xi}(\vartheta_{\frf,\frc,n}^{\alpha}(\frb))=(-1)^{\frac{d(d-1)}{2}}\frac{(\alpha-\ul{1})!}{\Omega^{\alpha}}\left(N\frc E^{0,\alpha}(1+r,0,\frf\frb^{-1},\Gamma)-E^{0,\alpha}(1+r,0,\frc^{-1}\frf\frb^{-1},\Gamma)\right).
\end{equation*}
Since we assumed $\rho=\delta_r$, the formula in the Proposition follows  from $\rho\ast\delta_1=\delta_r\ast\delta_1=\delta_{1+r}$. 
\end{proof}

\section{Distributions and cohomology classes}
In this section we discuss how cohomology classes give rise to distributions. Recall that $\Gamma\subset \sO_L^{\times}$ is a subgroup of  finite index. 

\begin{notation}
From now on, we consider $\cA$ over $\OCp$ via $\iota_p$.
\end{notation}
Let us write 
\begin{equation*}
    G=\cA[p^{\infty}]
\end{equation*}
for the $p$-divisible group of $\cA$ considered over $\sO_{\Cp}$ and $\cO(G)$ for its ring of functions. 
As $\Lie (G)=\Lie (\cA)$ one has  $\omega_G=\omega_\cA$. In particular, one has $\omega_G\isom \bigoplus_{\sigma\in \Sigma}\omega_G(-\sigma)$ and one can define $\TSym^{\alpha}(\omega_G)$ and $\Sym^{\alpha}(\Lie(G))$.
The dual Hodge--Tate pair, see \eqref{def:dual-HT-pair}, corresponding to $G$ is 
\[
	(W,T):=HT^\vee(G)=(\Lie (G)\otimes \Cp,T_p(G^\vee)).
\]
Theorem \ref{thm:UniformizationOfCharVar} gives over $\Cp$ a canonical isomorphism of adic spaces
\[
	\widehat{T}_W \cong G_\eta^\ad,
\]
and from Theorem \ref{thm:pAdicFourierMainTheoremW} it follows that 
\begin{equation}\label{eq:SigmaAnDistG}
	\cO(G_\eta^\ad)\cong D_W(T,\Cp).
\end{equation}

As in Notation \ref{not:diff-op} we view elements $\Xi\in \Sym^{\alpha}(\Lie(G))$ as differential operators on $\cO(G)$ and as polynomial functions $P_\Xi$ on $T=T_p(G^{\vee})$.  We write $j\colon T^\times \hookrightarrow  T$ for the inclusion of the open subset $T^\times \subseteq T$ and write
\[
	j^*\colon C_W^{an}(T,\Cp)\to C_W^{an}(T^\times,\Cp), \quad j_!\colon C_W^{an}(T^\times,\Cp)\to C_W^{an}(T,\Cp),
\]
for the restriction and the extension by zero map. For any locally constant function $\rho \colon T^\times\to \Cp$ the function $j_!\rho\colon T\to \Cp$ is locally constant as well, so that one can define the Fourier transform of $j_!\rho$ \ref{def:Fourier-transform} 
\[
	\widehat{j_!\rho}\colon\bigcup_n\widehat{T}_W(\Cp)[p^n]\to \Cp.
\]
Recall that one has $\cA(\CC)=\CC^\Sigma/\sO_L$. In particular, one has $\cA(\CC)[p^n]=p^{-n}\sO_L/\sO_L$. Using the canonical isomorphism
\[
	\bigcup_n p^{-n}\sO_L/\sO_L \cong\bigcup_n G(\OCp)[p^n]\cong\bigcup_n\widehat{T}_W(\Cp)[p^n],
\]
we will view $\widehat{j_!\rho}$ as a function on $\bigcup_n p^{-n}\sO_L/\sO_L$. The $W$-analytic distributions on $T^\times/E(\Gamma)$ can be identified with the continuous dual of the $\Gamma$-invariant functions
\begin{equation*}
    D_W(T^\times/E(\Gamma),\Cp):=(C^{\an}_W(T^\times,\Cp)^{\Gamma})'.
\end{equation*}
\begin{theorem}\label{thmB}
There exists a $\OCp$-linear map
\[
	H^{d-1}(\Gamma,\cO(G)\otimes\TSym^{\ul{1}}(\omega_\cA)\otimes \Sym^{\ul{1}}(\Lie(\cA)) \otimes \omega^d_{\cA})\to D_W(T^\times/E(\Gamma),\Cp), \quad \vartheta \mapsto \mu_\vartheta
\]
satisfying the following integration formula: For any 
locally constant $\rho:T^\times\to \ol{\QQ}\subseteq \C_p$ and $\Xi\in \Sym^{\alpha}(\Lie(G))$ such that $\rho\otimes\Xi$ is $\Gamma$-invariant, one has
	\begin{equation*}
		\int_{T^\times/E(\Gamma)} \rho\cdot j^*P_\Xi d{\mu}_{\vartheta}= \eval_{\widehat{j_!\rho}\otimes \Xi}(d^{\alpha-\underline{1}}\vartheta),
	\end{equation*}
	where $\eval_{\widehat{j_!\rho}\otimes \Xi}$ is the  evaluation map defined in Proposition \ref{prop:evaluation-map}.
\end{theorem}
\begin{proof}
We first construct the $\Cp$-linear map in the statement: The canonical $\Gamma$-equivariant isomorphism $\TSym^{\ul{1}}(\omega_\cA)\otimes \Sym^{\ul{1}}(\Lie(\cA))\cong \OCp$ gives
\begin{equation}\label{eq:contractSymTSym}
	H^{d-1}(\Gamma,\cO(G)\otimes\TSym^{\ul{1}}(\omega_\cA)\otimes \Sym^{\ul{1}}(\Lie(\cA)) \otimes \omega^d_{\cA})\cong H^{d-1}(\Gamma,\cO(G)\otimes \omega^d_{\cA})
\end{equation}
Passing to the generic fiber $G_\eta^\ad$ and applying \eqref{eq:SigmaAnDistG} gives a map
\begin{equation}\label{eq:FunctionsOnGtoDistributions}
	  H^{d-1}(\Gamma,\cO(G)\otimes \omega^d_{\cA}))\to H^{d-1}(\Gamma,D_W(T,\Cp)\otimes\omega^d_{\cA}).
\end{equation}
The map $j_!$ induces a restriction map on distributions
\begin{equation}\label{eq:RestrictDistributions}
	  H^{d-1}(\Gamma,D_W(T,\Cp)\otimes\omega^d_{\cA})\xrightarrow{(\cdot)|_{T^\times}} H^{d-1}(\Gamma,D_W(T^\times,\Cp)\otimes\omega^d_{\cA}).
\end{equation}
Observe that one has a $\Gamma$-equivariant $\Cp$-linear isomorphism
\begin{equation*}
	D_W(T^{\times},\Cp) \xleftarrow{\sim} D_W(T^\times,\Cp)\otimes\Sym^{\ul{1}}(\Lie(\cA)),\quad  P_{\Xi_{\ul{1}}}\cdot \mu \mapsfrom \mu\otimes \Xi_{\ul{1}},
\end{equation*}
which induces an isomorphism
\begin{equation}\label{eq:UntwistNSigma}
	H^{d-1}(\Gamma,D_W(T^{\times},\Cp)\otimes\omega^d_{\cA}) \xrightarrow{\sim} H^{d-1}(\Gamma,D_W(T^\times,\Cp)\otimes\Sym^{\ul{1}}(\Lie(\cA))\otimes\omega^d_\cA).
\end{equation}
The canonical $\Gamma$-equivariant pairing
\begin{equation*}
    C_W(T^\times,\C_p)\times D_W(T^\times,\Cp)\to \C_p
\end{equation*}
induces, by taking the cup-product, a pairing
\begin{equation*}
    H^{0}(\Gamma,C_W(T^\times,\C_p))\times H^{d-1}(\Gamma,D_W(T^\times,\Cp)\otimes \Sym^{\ul{1}}(\Lie(G))\otimes\omega^d_{\cA})\to 	H^{d-1}(\Gamma,\Sym^{\ul{1}}(\Lie(G))\otimes\omega^d_{\cA}).
\end{equation*}
If we combine this with the isomorphism in Proposition \ref{prop:cap-product} we get 
\begin{equation*}
     H^{0}(\Gamma,C_W(T^\times,\C_p))\times H^{d-1}(\Gamma,D_W(T^\times,\Cp)\otimes \Sym^{\ul{1}}(\Lie(G))\otimes\omega^d_{\cA})\to \C_p
\end{equation*}
which by duality gives
\begin{equation}\label{eq:eval}
    H^{d-1}(\Gamma,D_W(T^\times,\Cp)\otimes \Sym^{\ul{1}}(\Lie(G))\otimes\omega^d_{\cA})\to D_W(T^\times,\Cp)_\Gamma.
\end{equation}
The composition of \eqref{eq:contractSymTSym}, \eqref{eq:FunctionsOnGtoDistributions}, \eqref{eq:RestrictDistributions},  \eqref{eq:UntwistNSigma} and \eqref{eq:eval} gives the desired $\Cp$-linear map
\begin{align}\label{eq:Coh-to-Dist}
	H^{d-1}(\Gamma,\cO(G)\otimes\TSym^{\ul{1}}(\omega_\cA)\otimes \Sym^{\ul{1}}(\Lie(\cA))\otimes \omega^d_{\cA})&\to D_W(T^\times/E(\Gamma), \Cp),\\
	\notag \vartheta &\mapsto \mu_\vartheta.
\end{align}
The integration formula is a consequence of Corollary \ref{cor:dist-integration-formula}.
\end{proof}

\section{Proof of Theorem \ref{thmA}}
We recall the setup of theorem \ref{thmA}: $L$ is a totally imaginary field containing the maximal  CM field $K$, $p$ is a prime which is unramified in $L$, and $\Sigma$ is a CM type of $L$ induced from $K$, and $\frf\subseteq \sO_L$ is an ideal co-prime to $p$.

Further, we keep the previous notation, i.e., $\cA$ is an abelian variety over $R$ with a fixed isomorphism $\cA(\CC)\cong \CC^\Sigma/\sO_L$. We choose an $R$-basis $\omega(\cA)(\sigma)$ of $\omega_{\cA}(\sigma)$ for every $\sigma\in \Sigma$ and define the ($p$-adic) periods as in section \ref{sec:Periods}. We write $G=\cA[p^{\infty}]$ for the $p$-divisible group of $\cA$ over $\OCp$ and $\cO(G)=\prolim_n\cO(\cA[p^{n}])$. 

Let $\frc$ be an auxiliary ideal co-prime to $p\frf$ and recall that  we have defined in \ref{def:vartheta}  for $\alpha\geq \underline{1}$ and any ideal $\frb\subseteq\sO_L$ co-prime to $p\frf\frc$  cohomology classes
\begin{equation*}
    \vartheta_{\frf,\frc,n}^{\alpha}(\frb)\in H^{d-1}(\Gamma,\cO(\cA[p^{n}])\otimes\TSym^{\alpha}(\omega_\cA)\otimes\Sym^{\ul{1}}(\Lie(\cA))\otimes \omega^{d}_{\cA})).
\end{equation*}

\begin{definition} Let $\alpha\ge \ul{1}$ and $\frb$ an  ideal $\frb\subseteq\sO_L$ co-prime to $p\frf\frc$. We define
\begin{equation*}
    \vartheta^{\alpha}_{\frf,\frc}(\frb):=\vartheta^{\alpha}_{\frf,\frc,\Gamma}(\frb):=\prolim_n(\vartheta^{\alpha}_{\frf,\frc,n}(\frb))\in H^{d-1}(\Gamma,\cO(G)\otimes\TSym^{\alpha}(\omega_\cA)\otimes\Sym^{\ul{1}}(\Lie(\cA))\otimes\omega^{d}_{\cA}),
\end{equation*}
where we have used that the cohomology commutes with limits, as $\Gamma$ is finitely generated. If $\alpha=\ul{1}$ we write
\begin{equation*}
    \vartheta_{\frf,\frc}(\frb)=\vartheta^{\ul{1}}_{\frf,\frc}(\frb).
\end{equation*}
\end{definition}
\begin{proposition}\label{prop:theta-diff-compatibility-and-subgroups}
Let $d:\cO(G)\to \cO(G)\otimes\Omega^{1}_{\cA}$ be the differential on $\cO(G)$, where we have identified $\Omega^{1}_G\isom \Omega^{1}_{\cA}$. Then 
\begin{equation*}
    \vartheta_{\frf,\frc}^{\alpha}(\frb)=d^{\alpha-\ul{1}}\vartheta_{\frf,\frc}(\frb)
\end{equation*}
in $H^{d-1}(\Gamma,\cO(G)\otimes\TSym^{\alpha}(\omega_\cA)\otimes\Sym^{\ul{1}}(\Lie(\cA))\otimes\omega^{d}_{\cA})$.
\end{proposition}
\begin{proof}
This follows from  the compatibility of the differential on $\cO_\cA$ and $\cO(G)$. The second statement follows from \cite{Kings-Sprang} Corollary 2.29.
\end{proof}
Theorem \ref{thmB} assigns to $(-1)^{\frac{d(d-1)}{2}}\vartheta_{\frf,\frc}(\frb)$ a distribution $\widetilde{\mu}_{\frc,\frf}(\frb)\in D_W(T^\times/E(\sO_\frf^\times),\Cp)$,
where 
\[
	(W,T):=HT^\vee(G)=(\Lie (G)\otimes \Cp,T_p(G^\vee)).
\]
is the Hodge--Tate pair, see \eqref{def:dual-HT-pair}, corresponding to $G$. The fixed generator $\xi_p^\vee$ of the free $\sO_L\otimes \Zp$-module $T=T_p(G^\vee)$ gives an isomorphism $T\cong \sO_L\otimes \Zp$ and the $\Cp$-subvector space $W=\Lie(G)\otimes \Cp \subseteq \Hom(T,\Cp)$ is identified with the vector space $W(\Sigma)\subseteq \Hom(\sO_L\otimes \Zp,\Cp)$ defined in \eqref{eq:W-Sigma}. Using these identifications, we will view $\widetilde{\mu}_{\frc,\frf}(\frb)$ as a $\Sigma$-analytic distribution on $(\sO_L\otimes \Zp)^\times/E(\sO_\frf^\times)$, i.e., we get
\begin{equation}\label{eq:mu-tilde}
\widetilde{\mu}_{\frc,\frf}(\frb)\in D_{W(\Sigma)}((\sO_L\otimes \Zp)^\times/E(\sO_\frf^\times),\Cp)
\end{equation}
Our next goal is to compute the value of this distribution on certain $\Sigma$-analytic functions explicitly. For any locally constant function $\rho\colon (\sO_L\otimes \Zp)^\times \to \ol{\QQ}\subseteq\Cp$, the Fourier transform of the extension by zero gives a finitely supported function
\[
	\widehat{j_!\rho}\colon \sO_L\otimes \Qp/\Zp\to \ol{\QQ}, 
\]
i.e., a function in $\varinjlim_n C(p^{-n}\sO_L/\sO_L,\ol{\QQ})=C_c(\Qp\otimes \Lambda/\Lambda,R)$. Since $\frf$ and $\frb$ are co-prime to $p$, we have for $\Lambda=\frf\frb^{-1}$ canonical isomorphisms $p^{-n}\Lambda/\Lambda\cong p^{-n}\sO_L/\sO_L$. These isomorphisms allow us to view $\widehat{j_!\rho}\in C_c(\QQ\otimes\Lambda/\Lambda,R)$. Furthermore, we write $\delta_1\in C_c(\QQ\otimes\Lambda/\Lambda,R)$ for the characteristic function of $1+\Lambda$. Using the integration formulas from Theorem \ref{thmB}, we get:

\begin{proposition}\label{propB} The $\Sigma$-analytic distribution $\widetilde{\mu}_{\frc,\frf}(\frb)$ 
satisfies the following integration formulas:
For any locally constant $\rho \colon (\sO_L\otimes \Zp)^\times\to R\subseteq\Cp$ and $\alpha\in \ZZ_{> 0}^\Sigma$ such that the function $\rho(\cdot)(\cdot)^\alpha\colon \sO_L\otimes\Zp \to R$ is $\Gamma$-invariant, one has
	\begin{multline*}
		\frac{1}{\Omega_p^{\alpha}} \int_{(\sO_L\otimes \Zp)^\times/E(\sO_\frf^\times)} \rho(z)z^{\alpha} \, d\widetilde{\mu}_{\frf,\frc}(\frb)\\
= \frac{(\alpha-\ul{1})!}{\Omega^\alpha} \left(N\frc E^{0,\alpha}(\widehat{j_!\rho}\ast \delta_1,0; \frf\frb^{-1},\sO_{\frf}^\times)-E^{0,\alpha}(\widehat{j_!\rho}\ast \delta_1,0; \frf\frc^{-1}\frb^{-1},\sO_{\frf}^\times)\right),
	\end{multline*}
	where we view $\widehat{j_!\rho}\in C_c(\QQ\otimes\Lambda/\Lambda,R)$ for $\Lambda=\frf\frb^{-1}$ (resp. $\Lambda=\frf\frc^{-1}\frb^{-1}$) as explained above.
\end{proposition}
\begin{proof}
Our fixed basis $\omega(\cA)$ gives a basis $\LieA$ of $\Lie(\cA)$. By the definition of the $p$-adic periods, we have
\[
	\Lie(\cA)\hookrightarrow \Hom(T_p\cA^\vee, \Cp)=\Hom(\sO_L\otimes \Zp,\Cp), \quad \LieA(\sigma) \mapsto (1\mapsto \Omega_p(\sigma)).
\]
Hence, the interpolation formula in Theorem \ref{thmB} shows for $\Xi=(\LieA)^\alpha\in \Sym^\alpha(\Lie(\cA))$
\[
	\frac{(-1)^{\frac{d(d-1)}{2}}}{\Omega_p^{\alpha}} \int_{(\sO_L\otimes \Zp)^\times/E(\sO_\frf^\times)} \rho(z)z^{\alpha} \, d\widetilde{\mu}_{\frf,\frc}(\frb)= \eval_{\widehat{j_!\rho}\otimes \Xi}(d^{\alpha-\ul{1}}\vartheta_{\frf,\frc}(\frb))=\eval_{\widehat{j_!\rho}\otimes \Xi}(\vartheta^\alpha_{\frf,\frc}(\frb)).
\]
Now, the Proposition follows from Theorem \ref{thm:evaluation-formula} and the observation that $\eval_{\widehat{j_!\rho}\otimes \Xi}(\vartheta^\alpha_{\frf,\frc}(\frb))=\eval_{\widehat{j_!\rho}\otimes \Xi}(\vartheta^\alpha_{\frf,\frc,n}(\frb))$ for every $n\geq 1$ such that $\widehat{j_!\rho}$ is defined on $C(p^{-n}\sO_L/\sO_L,R)$.
\end{proof}

Let us introduce the local factor appearing in the statement of the main theorem \ref{thmA}. Let $\chi$ be a critical Hecke character of CM type $\Sigma$, infinity type $-\alpha\in \ZZ^\Sigma$ with $\alpha\geq \underline{1}$ and of conductor dividing $p^\infty \frf$. Let $\chi_{\mathrm{fin}}$ be the unique locally constant character
\begin{equation}\label{eq:chi-fin}
	\chi_{\mathrm{fin}} \colon (\sO_L\otimes \Zp)^\times \to \OCp^\times
\end{equation}
such that $\chi((\lambda))=\frac{\chi_{\mathrm{fin}}(\lambda)}{\lambda^\alpha}$ for all $\lambda\in \sO_L$ co-prime to $p\frf$. The isomorphism $(\sO_L\otimes \Zp)^\times\cong \prod_{\frp|p} \sO_{L,\frp}^\times$ allows us to write $\chi_{\fin}=\prod_{\frp|p} \chi_{\fin,\frp}$ with $\chi_{\fin,\frp}\colon \sO_{L,\frp}^\times\to \OCp^\times$. We define $F_\frp\colon \sO_{L,\frp} \to \OCp$ as follows
\begin{equation}\label{eq:F-ext-by-zero}
	F_\frp(x):=\begin{cases}
		\chi_{\fin,\frp}(x^{-1}) & \text{ if }\chi_{\fin,\frp} \text{ is non-trivial and }x\in\sO_{L,\frp}^\times, \\
		0 & \text{ if }\chi_{\fin,\frp} \text{ is non-trivial and }x\notin\sO_{L,\frp}^\times , \\
		1 & \text{ if }\chi_{\fin,\frp} \text{ is trivial},
	\end{cases}
\end{equation}
and set $F:=\prod_{\frp|p} F_\frp\colon \sO_L\otimes \Zp \to \OCp$.

\begin{definition}\label{def:Local-factor}
	Let $\chi$ be a critical algebraic Hecke character of conductor dividing $p^\infty \frf$ and infinity type $-\alpha$. For each prime ideal $\frp$ over $p$, let us write $m_\frp$ for the multiplicity of $\frp$ in the conductor $\cond(\chi)$ of $\chi$. Let us choose a decomposition $\prod_{\frp}\frp^{m_\frp}=(c)\frn$ with $\frn\subseteq \sO_L$ co-prime to $p\frf$ and $c\in L^\times$. We define
	\[
		\Local(\chi,\Sigma):=\frac{\widehat{F}(c^{-1})}{\chi(\frn)c^{-\alpha}}.
	\]
\end{definition}
It is not difficult to show that $\Local(\chi,\Sigma)$ is non-zero and independent of the decomposition $\prod_{\frp}\frp^{m_\frp}=(c)\frn$, and that $\Local(\chi,\Sigma)=1$ if $\chi$ is unramified at $p$, see \cite{Katz-CM}.

We can now finally prove the main Theorem:
\begin{proof}[Proof of Theorem \ref{thmA}]
	Let us write $G(p^\infty\frf):=\Gal(L(p^\infty\frf)/L)$, and choose ideals $\frb_1,\dots, \frb_h\subseteq \sO_L$ co-prime to $p\frf\frc$ which are a full set of representatives for $\Gal(L(\frf)/L)$ under the Artin map. By Proposition \ref{prop:decomp-Gal-distributions}, we have an isomorphism
	\begin{align}\label{eq:Distr-decomposition}
		D_{W(\Sigma)}(G(p^\infty\frf),\Cp)&\cong \bigoplus_{i=1}^h D_{W(\Sigma)}((\sO_L\otimes \Zp)^\times/E(\sO_\frf^\times),\Cp).
	\end{align}
 Proposition \ref{propB} gives us for each $i=1,\dots,h$ a unique $\Sigma$-analytic distributions $\widetilde{\mu}_{\frf,\frc}(\frb_i)$ with the interpolation property
	\begin{multline*}
		\frac{1}{\Omega_p^{\alpha}} \int_{(\sO_L\otimes \Zp)^\times/E(\sO_\frf^\times)} \rho(z)z^{\alpha} \, d\widetilde{\mu}_{\frf,\frc}(\frb_i)\\
		 = \frac{(\alpha-\ul{1})!}{\Omega^\alpha} \left(N\frc E^{0,\alpha}(\widehat{j_!\rho}\ast \delta_1,0; \frf\frb_i^{-1},\sO_{\frf}^\times)-E^{0,\alpha}(\widehat{j_!\rho}\ast \delta_1,0; \frf\frc^{-1}\frb_i^{-1},\sO_{\frf}^\times)\right).
	\end{multline*}
	For each $i=1,\dots, h$ we define $\mu_{\frf,\frc}(\frb_i)\in D_{W(\Sigma)}((\sO_L\otimes \Zp)^\times/E(\sO_L^\times),\Cp)$ by $\mu_{\frf,\frc}(\frb_i):=\mathrm{inv}^* \widetilde{\mu}_{\frf,\frc}(\frb_i)$, where $\mathrm{inv}\colon (\sO_L\otimes\Zp)^\times/E(\sO_L^\times) \xrightarrow{\sim} (\sO_L\otimes\Zp)^\times/E(\sO_L^\times)$ is the map $z\mapsto z^{-1}$, i.e., we have
	\[
		\int_{(\sO_L\otimes \Zp)^\times/E(\sO_\frf^\times)} f(z) d\mu_{\frf,\frc}(\frb_i):= \int_{(\sO_L\otimes \Zp)^\times/E(\sO_\frf^\times)} f(z^{-1}) d\widetilde{\mu}_{\frf,\frc}(\frb_i).
	\]
	We define the distribution $\mu_{\frf,\frc}\in D_{W(\Sigma)}(G(p^\infty\frf),\Cp)$ as the distribution corresponding to $(\mu_{\frf,\frc}(\frb_i))_{i=1}^h$ under the isomorphism \eqref{eq:Distr-decomposition}, and get
	\begin{align*}
		\int_{G(p^\infty\frf)}\chi(g)d\mu_{\frc,\frf}(g)=\sum_{i=1}^h \chi(\frb_i) \int_{(\sO_L\otimes \Zp)^\times/E(\sO_L^\times)} \chi_{\fin}^{-1}(z)z^{\alpha}d\widetilde{\mu}_{\frf,\frc}(\frb_i).
	\end{align*}
	This allows us to compute the integral in terms of Eisenstein series:
	\begin{multline}\label{eq:p-adicIntegral}
		\frac{1}{\Omega_p^\alpha}\int_{G(p^\infty\frf)}\chi(g)d\mu_{\frc,\frf} \\
		=\frac{(\alpha-\underline{1})!}{\Omega^\alpha}\sum_{i=1}^h\chi(\frb_i)\left(N\frc E^{0,\alpha}(\widehat{j_!\chi_\fin^{-1}}\ast \delta_1,0; \frf\frb_i^{-1},\sO_{\frf}^\times)-E^{0,\alpha}(\widehat{j_!\chi_\fin^{-1}}\ast \delta_1,0; \frf\frc^{-1}\frb_i^{-1},\sO_{\frf}^\times)\right).
	\end{multline}
	As before, let us denote the multiplicity of $\frp|p$ in $\cond(\chi)$ by $m_\frp$. The function $\widehat{j_!\chi_\fin^{-1}}$ is then supported on $(\prod_{m_\frp\geq 1}\frp^{-m_\frp})(\prod_{ m_\frp=0} \frp^{-1})\frf\frb_i^{-1}/\frf\frb_i^{-1}$. Using the maps
	\[
		1+ (\prod_{m_\frp\geq 1}\frp^{-m_\frp})(\prod_{ m_\frp=0} \frp^{-1})\frf\frb_i^{-1} \cong (\prod_{m_\frp\geq 1}\frp^{-m_\frp})(\prod_{ m_\frp=0} \frp^{-1})\frf\frb_i^{-1}  \twoheadrightarrow (\prod_{m_\frp\geq 1}\frp^{-m_\frp})(\prod_{ m_\frp=0} \frp^{-1})\frf\frb_i^{-1}/\frf\frb_i^{-1}
	\]
	we extend  $\widehat{j_!\chi_\fin^{-1}}$ to $1+ (\prod_{m_\frp\geq 1}\frp^{-m_\frp})(\prod_{ m_\frp=0} \frp^{-1})\frf\frb_i^{-1}$. With this convention, we get
	\begin{align*}
		&E^{0,\alpha}(\widehat{j_!\chi_\fin^{-1}}\ast \delta_1,0; \frf\frb_i^{-1},\sO_{\frf}^\times)=\left.\sum_{\widetilde{\lambda}\in \sO_\frf^\times\backslash (1+  (\prod_{m_\frp\geq 1}\frp^{-m_\frp})(\prod_{ m_\frp=0} \frp^{-1})\frf\frb_i^{-1})} \frac{\widehat{j_!\chi_\fin^{-1}}(\widetilde{\lambda})}{\widetilde{\lambda}^\alpha N(\widetilde{\lambda})^s}\right|_{s=0}.
	\end{align*}
	Let us now choose a decomposition
	\[
		\prod_{\frp} \frp^{m_\frp}=(c)\frn
	\]
	with $0\neq c\in 1+\frf$ and $\frn$ co-prime to $p$. We have
	\[
		1+  (\prod_{m_\frp\geq 1}\frp^{-m_\frp})(\prod_{ m_\frp=0} \frp^{-1})\frf\frb_i^{-1}=1+ c^{-1}\frb_i^{-1}\frn^{-1}(\prod_{ m_\frp=0} \frp^{-1})\frf=c^{-1}\left(1+\frb_i^{-1}\frn^{-1}(\prod_{ m_\frp=0} \frp^{-1})\frf\right),
	\]
	hence we can write every $\widetilde{\lambda}\in 1+ \prod_{\frp} \frp^{-m_\frp}\frf\frb_i^{-1}$ as $c^{-1}\lambda$ with $\lambda\in 1+\frb_i^{-1}\frn^{-1}(\prod_{ m_\frp=0} \frp^{-1})\frf$. Now, a straightforward computation shows
	\[
		\widehat{j_!\chi_\fin^{-1}}(c^{-1}\lambda)=\widehat{F}(c^{-1})\chi_{\fin}(\lambda)\prod_{\frp|p}\widehat{\Char_\frp}(\lambda),
	\]
	where $F=\prod_\frp F_\frp$ is the function defined in \eqref{eq:F-ext-by-zero}, and $\Char_\frp$ is the characteristic function of $\frp\otimes\Zp$ in $\sO_L\otimes \Zp$. It is not difficult to compute
	\[
		\widehat{\Char_\frp}(\lambda)=
		\begin{cases}
			1-N\frp^{-1} & \text{if } v_\frp(\lambda)\geq 0,\\
			-N\frp^{-1} & \text{if } v_\frp(\lambda)= -1,\\
			0 & \text{otherwise.}
		\end{cases}
	\]
	With this computation, we get
	\begin{align*}
		E^{0,\alpha}(\widehat{j_!\chi_\fin^{-1}}\ast \delta_1,0; \frf\frb_i^{-1},\sO_{\frf}^\times)=&\left.\frac{\widehat{F}(c^{-1})}{c^{-\alpha}}N(c)^s\sum_{\lambda \in \sO_\frf^\times \backslash (1+\frb_i^{-1}\frn^{-1}(\prod_{ m_\frp=0} \frp^{-1})\frf)} \frac{\chi_{\fin}(\lambda)}{\lambda^\alpha N(\lambda)^s}\prod_{\frp}\widehat{\Char_\frp}(\lambda)\right|_{s=0}\\
		=&\left.\frac{\widehat{F}(c^{-1})}{c^{-\alpha}}N(c)^s\sum_{\lambda \in \sO_\frf^\times \backslash (1+\frb_i^{-1}\frn^{-1}(\prod_{ m_\frp=0} \frp^{-1})\frf)} \frac{\chi((\lambda))}{N(\lambda)^s}\prod_\frp\widehat{\Char_\frp}(\lambda)\right|_{s=0}
	\end{align*}
	Next, we want to use the bijection
	\[
		\sO_\frf^\times \backslash (1+\frb_i^{-1}\frn^{-1}(\prod_{ m_\frp=0} \frp^{-1})\frf) \xrightarrow{\sim} \{ I\subseteq \prod_{ m_\frp=0} \frp^{-1}: I\in [\frn\frb_i] \}, \quad \lambda\mapsto \lambda \cdot \frn\frb_i
	\]
	to write relate this series to a partial $L$-value. Therefore, we extend $\widehat{\Char}_\frp$ to all fractional ideals as follows:
	\[
		\widehat{\Char}_\frp(I):=
		\begin{cases}
			1-N\frp^{-1} & \text{if } v_\frp(I)\geq 0,\\
			-N\frp^{-1} & \text{if } v_\frp(I)= -1,\\
			0 & \text{otherwise.}
		\end{cases}
	\]
	With this definition, we have $\widehat{\Char}_\frp(\lambda \frn\frc)=\widehat{\Char}_\frp(\lambda)$. We can now continue our computation:
	\begin{align*}
		&\sum_{i=1}^h \chi(\frb_i)E^{0,\alpha}(\widehat{j_!\chi_\fin^{-1}}\ast \delta_1,0; \frf\frb_i^{-1},\sO_{\frf}^\times)\\
		=&\left.\frac{\widehat{F}(c^{-1})}{c^{-\alpha}}N(c)^s\sum_{i=1}^h \chi(\frb_i)\sum_{\lambda \in \sO_\frf^\times \backslash (1+\frb_i^{-1}\frn^{-1}(\prod_{ m_\frp=0} \frp^{-1})\frf)} \frac{\chi((\lambda))}{N(\lambda)^s}\prod_\frp\widehat{\Char_\frp}(\lambda)\right|_{s=0}\\
		=&\left.\frac{\widehat{F}(c^{-1})}{c^{-\alpha}}N(c)^s\sum_{i=1}^h \chi(\frb_i)\sum_{\substack{0\neq I\subseteq \prod_{ m_\frp=0} \frp^{-1} \\ I\in [\frn\frb_i]}} \frac{\chi(\frn^{-1}\frb_i^{-1}I)}{N(\frn^{-1}\frb_i^{-1}I)^s}\prod_\frp\widehat{\Char_\frp}(I)\right|_{s=0}\\
		=&\left.\Local(\chi,\Sigma)\sum_{0\neq I\subseteq \prod_{ m_\frp=0} \frp^{-1} } \frac{\chi(I)}{NI^s}\prod_\frp\widehat{\Char_\frp}(I)\right|_{s=0}
	\end{align*}
	In this series, we may assume that $v_\frp(I)=0$ for every $\frp$ dividing the conductor of $\chi$ otherwise $\chi(I)=0$. Hence, we can write every $I\subseteq \prod_{m_\frp=0}\frp^{-1}$ appearing in the index uniquely as $I=\prod_{m_\frp=0}\frp^{n_\frp}I_0$ with $I_0\subseteq \sO_L$ prime to $p$ and $n_\frp\geq -1$ and get
	\begin{align*}
		\sum_{0\neq I\subseteq \prod_{ m_\frp=0} \frp^{-1}} \frac{\chi(I)}{NI^s}\prod_\frp\widehat{\Char_\frp}(I)&=\sum_{(I_0,p)=1} \frac{\chi(I_0)}{NI_0^s}\prod_{m_\frp=0} \sum_{n\geq -1}\frac{\chi(\frp)^n}{N\frp^{ns}}\widehat{\Char_\frp}(\frp^n)\\
		&=\sum_{(I_0,p)=1} \frac{\chi(I_0)}{NI_0^s}\prod_{m_\frp=0} \left(1-\frac{N\frp^s}{\chi(\frp)N\frp}\right) \sum_{n\geq 0}\frac{\chi(\frp)^n}{N\frp^{ns}}\\
		&=\sum_{(I_0,p)=1} \frac{\chi(I_0)}{NI_0^s}\prod_{m_\frp=0} \left(1-\frac{\chi(\frp^{-1})N\frp^s}{N\frp}\right) \left( 1-\frac{\chi(\frp)}{N\frp^s}\right)^{-1}\\
		&=\prod_{\frp} \left(1-\frac{\chi(\frp^{-1})N\frp^s}{N\frp}\right) \sum_{I_0} \frac{\chi(I_0)}{NI_0^s}\\
		&=\prod_{\frp} \left(1-\frac{\chi(\frp^{-1})N\frp^s}{N\frp}\right) L_\frf(\chi,s).
	\end{align*}
Together with \eqref{eq:p-adicIntegral}, this proves the interpolation formula
\[
	\frac{1}{\Omega_p^\alpha}\int_{G(p^\infty\frf)}\chi(g)d\mu_{\frc,\frf}=\frac{(\alpha-\underline{1})!}{\Omega^{\alpha}}\Local(\chi,\Sigma) (N\frc-\chi(\frc^{-1}))\prod_{\frp\mid p} \left( 1-\frac{\chi(\frp^{-1})}{N\frp} \right)L_\frf(\chi,0).
\]
Using the isomorphism of topological $\Cp$-algebras
\[
	D_{W(\Sigma)}(G(p^\infty\frf),\Cp)\cong\cO(\widehat{G(p^\infty\frf)}_{W(\Sigma)}),
\]
we may view $\mu_{\frc,\frf}$ as a function on the $\Sigma$-analytic character variety $\cO(\widehat{G(p^\infty\frf)})_{W(\Sigma)}$. For $\frc\subseteq \sO_L$ co-prime to $p\frf$ and $\frb_1,\dots,\frb_h$, let us write $V(\frc)$ for the vanishing locus of the function
\[
	\chi \mapsto N\frc -\chi(\frc^{-1})
\]
on $\widehat{G(p^\infty\frf)}_{W(\Sigma)}$. We obtain a rigid analytic function $L_{p,\frc}\in \cO(\widehat{G(p^\infty\frf)}_{W(\Sigma)}\setminus V(\frc))$ defined as
\[
	\chi \mapsto \frac{\mu_{\frc,\frf}(\chi)}{(N\frc-\chi(\frc^{-1}))},
\]
and for different ideals $\frc,\frc'$, we have $L_{p,\frc}=L_{p,\frc'}$ on the intersection of the open rigid analytic subvarieties $\widehat{G(p^\infty\frf)}_{W(\Sigma)}\setminus V(\frc)$ and $\widehat{G(p^\infty\frf)}_{W(\Sigma)}\setminus V(\frc')$. Now observe that the inverse norm character $N(\cdot)^{-1}$ is not $\Sigma$-analytic. Hence, we can find for any $\chi \in \widehat{G(p^\infty\frf)}_{W(\Sigma)}$ a $\frc\subseteq \sO_L$ such that $L_{p,\frc}$ is defined in a neighbourhood of $\chi$. This shows that the functions $L_{p,\frc}$ glue to a function $L_p\in \widehat{G(p^\infty \frf)}_{W(\Sigma)}$.
\end{proof}

\bibliographystyle{amsalpha}
\bibliography{p-adic-L-functions.bib}
\end{document}